\documentclass[11pt]{article} 
\usepackage{a4wide}
\usepackage{amsmath,amsfonts,amsthm, amssymb}
 
\newtheorem{theorem}{Theorem}
\newtheorem{remark}{Remark}
\newtheorem{definition}{Definition}

\newtheorem{proposition}{Proposition}
\newtheorem{lemma}{Lemma} 
\newtheorem{claim}{Claim}

\newcommand{\RR}{\mathbb{R}}

\newcommand{\NN}{\mathbb{N}}

\newcommand{\DD}{|D|}
\newcommand{\SC}{\mathcal{S}}
\newcommand{\px}{\partial_x}
\newcommand{\pxi}{\partial_\xi}

\def\im{\mathop{\rm Im}\nolimits}
\def\re{\mathop{\rm Re}\nolimits}

\allowdisplaybreaks

\begin{document}

\title{Local well-posedness and blow up  in the energy space for a class of $L^2$ critical  dispersion generalized Benjamin-Ono equations}
\author{C.E. Kenig$^{(1)}$, Y. Martel$^{(2), (3)}$ and L. Robbiano$^{(2)}$}
\date{\small 
$^{(1)}$ Department of Mathematics, University of Chicago,\\
5734 University avenue, Chicago, Illinois 60637-1514,\\
cek@math.uchicago.edu\\ \quad \\ 
$^{(2)}$Laboratoire de math\'ematiques de Versailles, UMR CNRS 8100,\\
 Universit\'e de Versailles Saint-Quentin-en-Yvelines,
\\
 45, av. des Etats-Unis,
 78035 Versailles cedex, France\\   
 yvan.martel@math.uvsq.fr, luc.robbiano@math.uvsq.fr
  \\ \quad \\ 
$^{(3)}$ Institut Universitaire de France}
\maketitle
\begin{abstract}
We consider a family of dispersion generalized Benjamin-Ono equations (dgBO)
$$
    u_t - \partial_x \DD^{\alpha} u + |u|^{2\alpha}  \px u = 0, \quad
    (t,x)\in \mathbb{R}\times \mathbb{R},
$$
where $\widehat {\DD^\alpha u} = |\xi|^\alpha \widehat u$ and  $1\leq \alpha\leq 2$.
These equations are critical with respect to the $L^2$ norm and global existence and interpolate between the modified BO equation ($\alpha=1$) and the critical gKdV equation ($\alpha=2$).

First, we prove local well-posedness in the energy space for $1<\alpha<2$, extending results in \cite{KPV91JAMS}--\cite{KPV} for the generalized KdV equations.

Second, we address the blow up problem in the spirit of \cite{MMjmpa,Me} concerning the critical gKdV equation, by studying rigidity properties of the (dgBO) flow in a neighborhood of the solitons. We prove that for $\alpha$ close to $2$,  solutions of negative energy close to solitons blow up in finite or infinite time in the energy space $H^{\frac \alpha 2}$.

The blow up proof requires both extensions to (dgBO) of  monotonicity results for local $L^2$ norms by pseudo-differential operator tools and perturbative arguments close to the (gKdV) case to obtain structural properties of the linearized flow around solitons.

\bigskip

\centerline{\bf R\'esum\'e}

\medskip

Nous consid\'erons une famille d'\'equations de Benjamin-Ono \`a dispersion g\'en\'eralis\'ee (dgBO)
$$
    u_t - \partial_x \DD^{\alpha} u + |u|^{2\alpha}  \px u = 0, \quad
    (t,x)\in \mathbb{R}\times \mathbb{R},
$$
o\`u $\widehat {\DD^\alpha u} = |\xi|^\alpha \widehat u$ et  $1\leq \alpha\leq 2$.
Ces \'equations sont critiques par rapport \`a la norme  $L^2$  et \`a l'existence globale et sont des interpolations entre  l'\'equation de Benjamin-Ono g\'en\'eralis\'ee critique ($\alpha=1$) et l'\'equation de Korteweg-de Vries g\'en\'eralis\'ee critique  ($\alpha=2$).

D'abord, nous montrons le caract\`ere bien pos\'e de ces \'equations dans l'espace d'\'energie
pour $1<\alpha<2$, \'etendant les r\'esultats de \cite{KPV91JAMS}--\cite{KPV} pour les \'equations de Korteweg-de Vries g\'en\'eralis\'ees.

Ensuite, nous \'etudions le ph\'enom\`ene d'explosion dans l'esprit de
 \cite{MMjmpa,Me} concernant l'\'equation de  gKdV critique, en \'etudiant les propri\'et\'es de rigidit\'e du flot de (dgBO) dans un voisinage des solitons. Nous montrons que pour  $\alpha$ proche de $2$,  
les solutions d'\'energie n\'egative proches des solitons explosent en temps fini ou infini dans l'espace d'\'energie
$H^{\frac \alpha 2}$.

La preuve de ce r\'esultat d'explosion n\'ecessite d'une part   l'adap\-tation \`a
 (dgBO) de r\'esultats de monotonie de normes $L^2$ locales par des m\'ethodes d'op\'e\-rateurs pseudo-differentiels et d'autre part des arguments de perturbation proche du cas (gKdV) pour obtenir des propri\'et\'es structurelles du flot lin\'earis\'e autour des solitons.
\end{abstract}
 
\section{Introduction}
We consider the following dispersion generalized Benjamin-Ono equations (dgBO)
\begin{equation}\label{dgBO}
    u_t - \partial_x \DD^{\alpha} u + |u|^{2\alpha}  \px u = 0, \quad
    (t,x)\in \mathbb{R}\times \mathbb{R},
\end{equation}
where  $\DD^\alpha$ is such that $\widehat {\DD^\alpha u} = |\xi|^\alpha \widehat u$ and  $1\leq \alpha\leq 2$.
Formally, the following three quantities are conserved for solutions
\begin{align}
& \int u(t,x)dx=\int u(0,x)dx,\\
\label{inv1}
&  M(t)  =\int u^2(t,x)dx =M(0), \\ 
& E(t) =\int \left( |\DD^{\frac \alpha 2} u|^2 - \frac {|u|^{2\alpha+2}}{(\alpha+1)(2 \alpha +1)}  \right) (t,x) dx=
E(0).\label{inv2}
\end{align}

Recall the scaling and translation invariances of equation \eqref{dgBO}:  if $u(t,x)$ is solution of \eqref{dgBO} then, for all  $\lambda_0>0$, $x_0\in \mathbb{R}$,
\begin{equation*}
  u_{\lambda_0,x_0}(t,x)=\lambda_0^{ - \frac 1{\alpha}} u(\lambda_0^{-(2+\frac 2{\alpha})} t,\lambda_0^{-\frac 2 \alpha} (x-x_0)) \text{ is also solution of \eqref{dgBO}.}
\end{equation*}
In particular, note that for any $\lambda_0>0$, $x_0\in \mathbb{R}$,
$
\|u_{\lambda_0,x_0}\|_{L^2} = \|u\|_{L^2},
$
which means that \eqref{dgBO} is a family of $L^2$ critical equations
interpolating between the critical Benjamin-Ono equation (also called   modified Benjamin-Ono equation)\begin{equation}\label{cgBO}
    u_t - \partial_x \DD  u +  u^2 \px u = 0, \quad
    (t,x)\in \mathbb{R}\times \mathbb{R},
\end{equation}
and the critical generalized Korteweg--de Vries equation
\begin{equation}\label{cgKdV}
    u_t + \partial_x^3 u +  u^4  \px u = 0, \quad
    (t,x)\in \mathbb{R}\times \mathbb{R}.
\end{equation}

\subsection{Local well-posedness in the energy space}

Recall that the local Cauchy problem is known to be well-posed in the energy space $H^{\frac \alpha 2}$  both for the critical (gKdV) equation -- see  Kenig, Ponce and Vega \cite{KPV} -- 
and for the critical (BO) equation -- see Kenig and Takaoka \cite{KT}.
The first objective of this paper is to present a local Cauchy theory for \eqref{dgBO} in the energy space $H^{\frac \alpha 2}$  for $1<\alpha<2$. 

\begin{theorem}[Local well-posedness in the energy space]\label{TH1}
Let $1< \alpha < 2$ and $A>0$.
Let $u_0\in H^{\frac \alpha 2}$ be such that $\|u_0\|_{H^{\frac \alpha 2}}\leq A$.
Then there exists a unique solution   $u\in C([0,T], H^{\frac \alpha 2})\cap Z_T$ of 
\begin{equation}\label{eq:dgBOpm}
\left\{  \begin{split} & u_t - \partial_x \DD^{\alpha} u \pm |u|^{2\alpha}  \px u = 0, \quad
    (t,x)\in \mathbb{R}\times \mathbb{R},\\
    & u(t=0)=u_0,\quad x\in \mathbb{R},
\end{split} \right.
\end{equation}
where $T=T(A)>0$.
Moreover, the map $u_0\mapsto u \in C([0,T], H^{\frac \alpha 2})\cap Z_T$ is continuous.
\end{theorem}

Theorem \ref{TH1} is proved by a contraction argument in $Z_T$, see
  the proof of Theorem \ref{TH1} for the definition of this functional space.
The linear estimates, mainly taken from \cite{KPVIUJ} and \cite{KPV},  are gathered in Lemma \ref{le:linear}.

\begin{remark}
Together with Theorem \ref{TH1}, we obtain in this paper a property of weak continuity of the flow of equation \eqref{dgBO} in the energy space, see Theorem \ref{TH3}. 
See also \cite{MMjmpa} and \cite{CK} for the cases $\alpha=1,$ $2$.
\end{remark}

In this paper, by solutions of \eqref{dgBO}, we mean $H^{\frac \alpha 2}$ solutions in the sense of Theorem \ref{TH1}. For such solutions, it follows from standard arguments that the two quantities $M(u(t))$ and $E(u(t))$ defined in \eqref{inv1}, \eqref{inv2} are conserved as long as the solution exists
(see also Remark \ref{RE:bip}).

\subsection{Blow up in finite or infinite time}

The second  objective of this paper is to study  global well-posedness versus blow up for equations \eqref{dgBO}, i.e. in the focusing case. 
Recall that   equations \eqref{dgBO} are critical with respect to global well-posedness in the following sense. For fixed $1\leq \alpha\leq 2$, the power $2\alpha+1$ of the nonlinearity in \eqref{dgBO} is the smallest power for which blow up is possible in the energy space, whereas from the critical Gagliardo-Nirenberg inequality 
\begin{equation}\label{gn}
\int |u|^{2\alpha+2}\leq C_\alpha \left(\int |D^{\frac \alpha 2} u|^2 \right)
\left(\int u^2\right)^{\alpha},
\end{equation}
 it is a standard observation that small (in $L^2$) solutions of \eqref{dgBO} are global and bounded from Theorem \ref{TH1}.
Note that inequality \eqref{gn} is easily proved using Fourier analysis and scaling arguments. See Proposition \ref{le:Q} for the value of the best  constant in \eqref{gn}, related to soliton solutions of \eqref{dgBO}.
 
Following Martel and Merle \cite{MMjmpa} and Merle \cite{Me} concerning the critical  (gKdV) equation, we look for blow up solutions close to the soliton family, which we introduce now.
We call soliton any traveling wave solution $u(t,x)=Q_{\lambda_0}(x-x_0-\lambda_0^{-2} t)$ of the equation,
with $\lambda_0>0$, $x_0\in \mathbb{R}$, $Q_{\lambda_0}(x)=\lambda_0^{- \frac 1{\alpha}} Q(\lambda_0^{- \frac 2 \alpha} x)$ and where $Q$  solves:
\begin{equation}\label{eq:Q}
\DD^\alpha Q + Q - \tfrac 1{2\alpha+1} {Q^{2\alpha+1}} =0, \quad
Q\in H^{\frac \alpha 2},\quad Q>0.
\end{equation}
For the critical (gKdV) case ($\alpha=2$), it follows from standard ODE arguments  that there exists a unique (up to translations)   solution of \eqref{eq:Q}, which is
\begin{equation}\label{explicitQ}
Q(x)=\frac {15^{1/4}} {{\rm cosh}^{1/2} (2x)}.
\end{equation}
Moreover, Weinstein \cite{We83} proved that the function $Q$ provides the best constant in estimate \eqref{gn} for $\alpha=2$.
For general values of $1\leq \alpha<2$,  existence of a positive even solution of \eqref{eq:Q} is known by variational arguments, see Weinstein \cite{We85,We87} and Proposition \ref{le:Q} of the present paper. 
Such a solution is called a ground state of \eqref{eq:Q}. However,  no explicit formula is known for $Q$ and uniqueness is an open problem. Note in particular that the striking uniqueness proof of Amick and Toland \cite{AT} for the Benjamin-Ono equation (BO)
\begin{equation}\label{BO}
    u_t - \partial_x \DD  u + u  \px u = 0, \quad
    (t,x)\in \mathbb{R}\times \mathbb{R},
\end{equation}
does not seem to apply to other than   quadratic nonlinearity.

For $\alpha<2$ close enough to $2$ (i.e. when  the model is close in some sense to the critical generalized KdV equation),  by perturbative arguments, we are able to extend some properties of the (gKdV) case. In particular, we prove uniqueness in some sense of the ground state of \eqref{eq:Q}.
We also prove in this framework  a crucial rigidity property of the linearized flow around soliton (hereafter called linear Liouville property), see Proposition \ref{pr:close2}.

From this linear Liouville property and monotonicity properties of local $L^2$ quantities (proved in Section 4), we can extend the main results in \cite{MMjmpa} and  \cite{Me}  to the dispersion generalized BO equation \eqref{dgBO} for $\alpha$ close to $2$. In particular, we claim the following result of finite or infinite time blow up.

\begin{theorem}[Blow up in finite or infinite time]\label{TH2}
There exists $\alpha_{0} \in [1,2)$ such that for all $\alpha \in (\alpha_0,2)$, the following holds.
\begin{itemize}
\item[\rm (i)]  There exists a unique even positive solution $Q$ of \eqref{eq:Q} which minimizes the constant $C_\alpha$ in \eqref{gn}.
\item[\rm (ii)] There exists $\beta_{0}>0$ such that if $u(t)$ is an $H^{\frac \alpha 2}$ solution of \eqref{dgBO}  such that
$$
E(u(0))<0 \quad \text{and} \quad \int u^2(0) \leq \int Q^2 + \beta_{0},
$$
then $u(t)$ blows up in finite or infinite time in $H^{\frac \alpha 2}$.
\end{itemize}
\end{theorem}

Note that from the Gagliardo--Nirenberg's inequality with best constant (see Proposition~\ref{le:Q}) $E(u(0))<0$ implies that 
$\int u(0)^2 > \int Q^2$. Therefore, we prove blow up in finite or infinite time for any $u(0)$ such that
$E(u(0))<0$, $\int Q^2 < \int u^2(0) \leq \int Q^2 + \beta_{0}$, which is a large class of initial data close to $Q$ up to the invariances of the equation (see Lemma \ref{claim10}).

\medskip

As mentionned before, Theorem \ref{TH2} (i) is obtained by perturbative arguments. Further natural properties of $Q$ -- such as decay at infinity -- are also presented in Proposition \ref{le:Q}.

Theorem \ref{TH2} (ii) is the extension to \eqref{dgBO} of the main result in
\cite{Me} following the same strategy based on rigidity properties of the nonlinear flow around solitons. First, we prove a nonlinear Liouville property around the soliton as a consequence of the linear Liouville property above discussed. See Section 5 where we extend the main results of \cite{MMjmpa}. 
Then, in Section 6, we prove blow up in the sense of Theorem \ref{TH2} by a contradiction argument, using the nonlinear Liouville property and  the additional  invariant $\int u(t)=\int u(0)$,  as in \cite{Me}.
\medskip

We now discuss how techniques involved in \cite{MMjmpa} and \cite{Me} have been extended to eq. \eqref{dgBO}.
\begin{itemize}
\item  Monotonicity properties of local $L^2$ quantities. These arguments were developed in \cite{MMjmpa} and \cite{Me} in order to study the variation in time of the $L^2$ norm of the solution in various regions of space (on the left or on the right, in some sense, to the soliton). For the critical (gKdV) equation, these monotonicity arguments are  mainly based on the Kato identity  and refined estimates on the nonlinear term in this identity. For Benjamin-Ono type equations, such localization arguments are subtle to adapt due to the nonlocal character of the linear operator. Such $L^2$ monotonicity arguments were developed in \cite{KM} to prove asymptotic stability of the solitons for the (BO) equation, but the arguments in \cite{KM} seem to work only for the operator $\DD$, i.e. for $\alpha=1$ in \eqref{dgBO}. In the present paper, we extend these results to any $\alpha \in (1,2)$ using  tools from pseudo-differential calculus. Section 4 is devoted to these arguments.

\item Weak continuity of the flow. In addition to  the local Cauchy theory, we need the weak continuity of the flow of \eqref{dgBO}   in several key limiting arguments. See Theorem \ref{TH3}.

\item The linear Liouville property.  The linear Liouville property is   obtained by perturbation of the (gKdV) case, originally treated in \cite{MMjmpa}. In this paper, we rely on the approach of \cite{yvanSIAM}.
\end{itemize}
It follows from the arguments of this paper that Theorem \ref{TH2} holds true for any $1<\alpha<2$ provided the linear liouville property is assumed. Indeed, it is the only part in the proof of Theorem \ref{TH2} where we need perturbative arguments close to the (gKdV) case.  In particular, the monotonicity arguments and the overall strategy work for any $1< \alpha \leq 2$. 

Remark finally that in addition to \cite{MMjmpa} and \cite{Me}, two further works  (\cite{MM1} and \cite{MM2}) provide refined information about the blow up phenomenon for the critical (gKdV) equation close to the soliton family. Indeed, in \cite{MM1}, the soliton $Q$ is found to be the universal blow up profile in the context of Theorem \ref{TH2}. The proof is based on an additional rigidity property of the (gKdV) flow around solitons in a blow up  regime. Finally, \cite{MM2} proves blow up in finite time, together with an upper estimate on the blow up rate, provided that the initial data has some space decay. However, note that the blow up problem for the critical (gKdV) equations is   not yet completely understood, in particular the blow up rate. The case of the nonlinear Schr\"odinger equation is by now much better known, see Merle and Rapha\"el \cite{MR3, MR4, MR5} and references therein.
For simplicity and brevity, we do not try here to extend results of \cite{MM1} and \cite{MM2} to (dgBO) equation.

\subsection{Plan of the paper}

The paper is organized as follows. In Section 2, we prove Theorem \ref{TH1}. In Section 3, we study the stationary problem \eqref{eq:Q} in the general case $1\leq \alpha \leq 2$ and obtain further properties in the perturbative case where $\alpha$ is close to $2$. In Section 4, we present  $L^2$ monotonicity properties for the model \eqref{dgBO} for all $\alpha\in (1,2)$. In Section 5, we deal with solutions close to a (bounded) soliton and finally in Section~6, we prove Theorem \ref{TH2}, i.e. for $\alpha$ close to $2$, blow up in finite or infinite time for negative energy solutions close to  solitons.

\medskip

\noindent\textbf{Acknowledgments.} Part of this work was done while the second author was visiting the University of Chicago. He would like to thank the Department of Mathematics for its hospitality.
The second author was partly supported by the French Agence Nationale de la Recherche, ANR ONDENONLIN.

\section{Local well-posedness  in the energy space}
\subsection{Proof of Theorem \ref{TH1}}

We denote the Fourier transform by $\mathcal{F}(f)(\xi)=\hat f(\xi) = \int e^{-i x \xi} f(x) dx$.

We introduce the group $W_\alpha(t)$ defined by
$$
\mathcal{F}(W_{\alpha}(t)f)(\xi)
=e^{it(|\xi|^\alpha \xi)} \widehat f(\xi),\quad 1<\alpha<2.
$$
Then, we claim (or recall) the following linear estimates (we use classical notation from \cite{KPV}).

\begin{lemma}\label{le:linear}
For $0<T<1$, there exist $C>0$ such that, for all $u_0 \in L^2$ , then
\begin{align*}
& (i) \qquad \sup_{t} \|W_{\alpha}(t)u_0\|_{L^2}\leq C \|u_0\|_{L^2};\\
& (ii) \qquad \||D|^{\frac \alpha 2} W_{\alpha}(t)u_0\|_{L^\infty_xL^2_t}
\leq C\|u_0\|_{L^2}; \\ & (iii) \qquad \textrm{ for all } u_0\in H^{\alpha/2},\ \|W_{\alpha}(t)u_0\|_{L^\infty_xL^2_T} \leq CT^{1/2}
\|u_0\|_{H^{\frac \alpha 2}};\\
& \textrm{ for }0\le \beta< \alpha/2, \textrm{ there exists }\gamma>0 \textrm{ such that, }\\
&(iv) \qquad \||D|^{\beta} W_{\alpha}(t)u_0\|_{L^\infty_xL^2_T}\leq CT^{\gamma}\|u_0\|_{L^2}; \\
&(v)\qquad  \textrm{ for all } u_0\in H^{\alpha/2},\ \|\partial_xW_\alpha(t)u_0\|_{L^\infty_xL^2_T}\le CT^\gamma\| u_0\|_{H^{\alpha/2}};\\
&(vi)\qquad  \textrm{ for all } u_0\in H^{\beta^+},\textrm{ where }\beta^+>\frac34-\frac\alpha4, \  \|W_\alpha(t)u_0\|_{L^{2\alpha}_xL^\infty_T}\le C\| u_0\|_{H^{\beta+}};\\
& \textrm{ for all } h\in L^1_xL^2_T,\\
&(vii)\qquad \||D|^\alpha \int_0^tW_\alpha(t-t')h(t')dt'\|_{L^\infty_x L^2_t}\le C\|h\|_{L^1_x L^2_T};\\
&(viii)\qquad \sup_{0<t<T}\||D|^{\alpha/2}\int_0^tW_\alpha(t-t')h(t')dt'\|_{L^2_x}\le C\|h\|_{L^1_xL^2_T};\\
& \textrm{ for }0\le\beta<\alpha, \textrm{ there exists }\gamma>0 \textrm{ such that, }\\
&(ix)\qquad \||D|^\beta\int_0^tW_\alpha(t-t')h(t')dt'\|_{L^\infty_xL^2_T}\le CT^\gamma \|h\|_{L^1_xL^2_T};\\
&\textrm{ there exists }\gamma>0 \textrm{ such that, }\\
&(x)\qquad \sup_{0<t<T}\|\int_0^tW_\alpha(t-t')h(t')dt'\|_{L^2_x}\le CT^\gamma\|h\|_{L^1_xL^2_T};\\
&\textrm{ for all } h\in L^1_xL^2_T \textrm{ such that } \partial_xh\in L^1_xL^2_T,\\
&(xi)\qquad \|\int_0^t W_\alpha(t-t')\partial_xh(t')dt'\|_{L^{2\alpha}_xL^\infty_T}\le C\|\partial_xh\|_{L^1_xL^2_T} 
\end{align*}
\end{lemma}

\begin{proof}
(i) is the classical conservation law, and (ii) (sharp Kato smoothing effect) is proved in \cite{KPV91JAMS} Lemma 2.1.

By Sobolev embedding 
\begin{equation}\label{eq:linear1}
 |W_\alpha(t)u_0(x)|^2\le C\| W_\alpha(t)u_0\|_{H^{\frac\alpha2}}^2\le C\| u_0\|_{H^{\frac\alpha2}}^2
\end{equation}
Integrating \eqref{eq:linear1} with respect $t$, we obtain  (iii).

To prove  (iv), we first write $u_0=u_{0,1}+u_{0,2}$ where $\hat u_{0,1}(\xi)=\chi_{|\xi|\le M}\hat u_0(\xi)$, for $M>0$ to be chosen.
Consider $\|\DD^\beta W_\alpha (t)u_{0,2}\|_{L^\infty_xL^2_t}$ and let $v_{0,2}$ such that $\hat v_{0,2}(\xi)=\frac{|\xi|^\beta}{|\xi|^{\alpha/2}}\hat u_{0,2}(\xi)$, so that we have $|D|^{\alpha/2}W_\alpha(t)v_{0,2}=|D|^\beta W_\alpha (t)u_{0,2}$. Using (ii), we see that, since $\beta< \alpha/2$,
\begin{equation}\label{eq:linear2}
 \| |D|^\beta W_\alpha (t)u_{0,2}\|_{L^\infty_xL^2_T}\le C\|v_{0,2}\|_{L^2}\le CM^{\beta-\alpha/2}\| u_0\|_{L^2}.
\end{equation}

Consider now $|D|^\beta u_{0,1}$. Then, $\||D|^\beta u_{0,1}\|_{H^{\alpha/2}}\le CM^{\beta+\alpha/2}\| u_0\|_{L^2}$. From (iii),

\begin{equation}\label{eq:linear3}
 \||D|^\beta W_\alpha(t)u_{0,1}\|_{L^\infty_xL^2_T}\le CT^{1/2}M^{\beta+\alpha/2}\| u_0\|_{L^2}.
\end{equation}
Choosing $M^{\alpha}=T^{-1/2}$, estimate  (iv) follows from \eqref{eq:linear2} and \eqref{eq:linear3}.

Writing $|D|=|D|^{1-\alpha/2}|D|^{\alpha/2}$ we use (iii) and (iv) and the fact that $1-\alpha/2<\alpha/2$ to prove (v).

In \cite{KPV91JAMS}, proof of Theorem 2.7, page 332, it is proved that if $|\xi|\simeq 2^k$ (or $|\xi|\lesssim 1 $ for $k=0$) we have, for $\hat u_0$ with that support
\begin{equation}
 \| W_\alpha(t)u_0\|_{L^2_xL^\infty_T}\le C2^{k(\alpha+1)/4}\|u_0\|_{L^2}.
\end{equation}
Also in \cite{KPVIUJ}, Theorem 2.5,  it is proved  that
\begin{equation}
 \|W_\alpha(t)u_0\|_{L^4_xL^\infty_T}\le C\| \DD^{1/4}u_0\|_{L^2}.
\end{equation}
Write now $\frac1{2\alpha} =\frac\theta2+\frac{1-\theta}{4}$ then $\theta=\frac2\alpha -1 \in(0,1)$, by interpolation, we get,
\begin{equation}\label{eq:linear4}
 \|W_\alpha(t)u_0\|_{L^{2\alpha}_xL^\infty_T}\le C2^{k(1-\theta)/4}2^{k(1+\alpha)\theta/4}\|u_0\|_{L^2}=C2^{k(3-\alpha)/4}\|u_0\|_{L^2}.
\end{equation}
which implies (vi). Note that \eqref{eq:linear4} is more precise for $\hat u_0$ supported in $|\xi|\simeq 2^k$.

Estimates (vii) and (viii) are proved in a similar way as (3.8) and (3.7) in \cite{KPV}. We omit their proofs.

Let $\theta\in {\cal C}_0^\infty $, $\theta\equiv1$ for $ |\xi|\le 1$, $\theta_M(\xi)=\theta(\xi/M)$, $\psi_M(\xi)=1-\theta_M(\xi)$ where $M\ge 1$. Write $h=h_{1,M}+h_{2,M}$, where $\hat h_{1,M}(t,\xi)=\theta_M(\xi)\hat h(t,\xi)$. Write $\tilde h_{2,M}$ by $(\tilde h_{2,M})^\wedge(t,\xi)=\frac{|\xi|^\beta}{|\xi|^\alpha}\hat h_{2,M}(t,\xi)$.
Thus $|D|^\beta\int_0^tW_\alpha(t-t')h_{2,M}(t')dt'=|D|^\alpha\int_0^tW_\alpha(t-t')\tilde h_{2,M}dt'$. Let $\hat\eta_M(\xi)=\frac{|\xi|^\beta}{|\xi|^\alpha}\psi_M(\xi)$. 
Using a dyadic partition of unity in frequency space and Bernstein inequality, we claim  $\int |\eta_M|\le CM^{\beta-\alpha}$.  

Thus 
\begin{equation}\label{eq:linear5}
 \|\tilde h_{2,M}\|_{L^1_xL^2_T}\le CM^{\alpha-\beta}\|h\|_{L^1_xL^2_T}
\end{equation}
so that by (vii),
\begin{equation}\label{eq:linear5bis}
 \||D|^\beta\int_0^tW_\alpha(t-t')h_{2,M}(t')dt'\|_{L^\infty_xL^2_T}\le C M^{\alpha-\beta}\|h\|_{L^1_xL^2_T}.
\end{equation}
Next, we consider $|D|^\beta\int_0^tW_\alpha(t-t')h_{1,M}(t')dt'$. Then, let us define $\hat \mu_M(\xi)=|\xi|^\beta\langle\xi\rangle\theta_M(\xi)$ where $\langle\xi\rangle^2=1+|\xi|^2$. Then, $\|\mu_M\|_{L^2}\le CM^{\beta+3/2}$. Morever, for a fixed $t$, we have by Sobolev embedding, 
\begin{equation}\label{eq:linear6}
\begin{array}{ll}
 \left||D|^\beta\int_0^tW_\alpha(t-t')h_{1,M}(t')dt'\right|&\le C\|\int_0^tW_\alpha(t-t')\langle D\rangle |D|^\beta h_{1,M}(t')dt'\|_{L^2_x}\\
&\le C\int_0^T\|\mu_M*h(t')\|_{L^2_x}dt'\\
&\le CM^{\beta+3/2}\int_0^T\|h(t')\|_{L^1_x}dt'=CM^{\beta+3/2}\|h\|_{L^1_xL^1_T}\\
&\le CT^{1/2}M^{\beta+3/2}\|h\|_{L^1_xL^2_T}.
\end{array}
\end{equation}
Hence, pick $M$ so that $M^{\alpha+3/2}=T^{-1/2}$, \eqref{eq:linear5bis} and  \eqref{eq:linear6} prove estimate (ix).

We obtain (x) by duality to the case $\beta=0$ of  (iv). Let $g\in L^1_xL^2_T$, $\|g\|_{L^1_xL^2_T}=1$. Then
\begin{equation}
 \int_0^T\int W_\alpha(t)u_0(x)\overline{g(t,x)}dxdt=\int \int_0^Tu_0(x)\overline{W_\alpha(-t)g(t,x)}dtdx.
\end{equation}
So estimate  (iv) is equivalent to
\begin{equation}
 \|\int_0^TW_\alpha(-t')g(t',x)dt'\|_{L^2_x}\le CT^\gamma\|g\|_{L^1_xL^2_T}.
\end{equation}
Fix $0<t<T$, let $g(t',x)=\chi_{[0,t]}(t')h(t,x)$, then
\begin{equation}
 \|\int_0^tW_\alpha(-t')h(t',x)dt'\|_{L^2_x}\le CT^\gamma\|h\|_{L^1_xL^2_T}.
\end{equation}
Apply now $W_\alpha(t)$ to the left hand side, which is an isometry in $L^2$, to obtain (x).

Let $P_k$ be a projection on frequencies $\simeq 2^k$ (or $\le 1$ for $k=0$), which is smooth on Fourier Transform side. Consider 
\begin{equation}
 \begin{array}{ll}
 &T_kh(x,t)=\int_0^tW_\alpha(t-t')P_k\partial_xh(\cdot,t')dt';\\[3pt]
&\tilde T_kh(x,t)=\int_0^TW_\alpha(t-t')P_k\partial_xh(\cdot,t')dt'.
\end{array}
\end{equation}
By (vi), localization in frequencies and (viii) we have, for $\frac 34 - \frac \alpha 4 < \beta^+ < \frac \alpha 2$,
\begin{equation}
 \begin{array}{ll}
    \|\tilde T_kh\|_{L^{2\alpha}_xL^\infty_T}&=\|W_\alpha(t)\int_0^TW_\alpha(-t')P_k\partial_xh(\cdot,t')dt'\|_{L^{2\alpha}_xL^\infty_T}\\[5pt]
&\le C2^{k(\beta^+-\alpha/2)}\|\int_0^T|D|^{\alpha/2}W_\alpha(-t')P_k\partial_xh(\cdot,t')dt'\|_{L^2_x}\\
&\le C 2^{k(\beta^+-\alpha/2)}\|\int_0^T|D|^{\alpha/2}W_\alpha(T-t')P_k\partial_xh(\cdot,t')dt'\|_{L^2_x}\\
&\le C 2^{k(\beta^+-\alpha/2)}\| \partial_x h\|_{L^1_xL^2_T}.
\end{array}
\end{equation}\label{eq:linear7}
Using the version of the Christ and Kiselev's lemma in Molinet and Ribaud \cite{MR}, Lemma 3, we obtain,
\begin{equation}
 \| T_kh\|_{L^{2\alpha}_xL^\infty_T}\le C  2^{k(\beta^+-\alpha/2)}  \|\partial_xh\|_{L^1_xL^2_T}.
\end{equation}
The sum of right side of \eqref{eq:linear7} being convergent, (xi) follows.
\end{proof}

We are now ready for our well-posedness result in the energy space, Theorem \ref{TH1}.
\begin{proof}[Proof of  Theorem \ref{TH1}]
Let $Z_T$ be the space defined by the maximum of the following norms, $$\sup_{0\le t\le T} \|u\|_{H^{\alpha/2}}, \ \||D|^\alpha u\|_{L^\infty_xL^2_T}, \ T^{-\gamma}\|u\|_{L^\infty_x L^2_T}, \ T^{-\gamma}\|\partial_x u\|_{L^\infty_x L^2_T}, \ \|u\|_{L^{2\alpha}_xL^\infty_T},$$
for some $\gamma>0$ to be chosen.

Fix $u_0\in H^{\alpha/2}$, $\|u_0\|_{H^{\alpha/2}}\le A$. 
For $R$, $T$ to be determined, let $B_{R,T}=\{ v\in Z_T,\ \|v\|_{Z_T}\le R\}$. Let
\begin{equation}
 \Phi_{u_0}(v)=W_\alpha(t)u_0\pm \int_0^tW_\alpha(t-t') (|v|^{2\alpha}\partial_xv)(t')dt'.
\end{equation}

We will show that, given $A$, we can find $R$, $T$ such that $\Phi_{u_0}(v): B_{R,T}\to B_{R,T}$ and is a contraction there. First, note that (i), (ii), (iii), (iv) (with $\beta=1/2$), (v), (vi) show that $\|W_\alpha(t)u_0\|_{Z_T}\le CA$, for some $\gamma>0$.

Now, we work on the Duhamel term. It is easy to see that, using (vii), (viii), (ix) (with $\beta=0$ and $\beta=1$), (x), (xi), we have 
\begin{equation}
 \begin{array}{ll}
 \|\int_0^tW_\alpha(t-t')|v|^{2\alpha}\partial_xvdt'\|_{Z_T}&\le C\left\{ \||v|^{2\alpha}\partial_xv\|_{L^1_xL^2_T}+\||v|^{2\alpha}v\|_{L^1_xL^2_T}\right\}\\[6pt]
&\le CT^{\gamma}\|v\|^{2\alpha}_{L^{2\alpha}_xL^\infty_T}\|v\|_{Z_T},
\end{array}
\end{equation}
for some $\gamma>0$.
We now choose $R=2CA$ and $T$ so that $C(2CA)^{2\alpha}T^\gamma\le CA=\frac 12 R$, which gives $\Phi_{u_0}:B_{R,T}\to B_{R,T}$.

For the contraction property, we estimate,
\begin{equation}
 \left| |v|^{2\alpha}\partial_xv-|w|^{2\alpha}\partial_xw\right|\le \left| (|v|^{2\alpha}-|w|^{2\alpha})\partial_xw\right|+\left| |v|^{2\alpha}\partial_x(v-w)\right|
\end{equation}
since $\left| |v|^{2\alpha}-|w|^{2\alpha}\right|\le C|v-w|(|v|^{2\alpha-1}+|w|^{2\alpha-1})$ and $\alpha>1$, this allows to conclude the proof (we argue similarly for $\left| |v|^{2\alpha}v-|w|^{2\alpha}w\right|$).
\end{proof}

\begin{remark}\label{RE:bip}
From Theorem \ref{TH1}, it follows that for any initial data in $H^{\frac \alpha 2}$, we can define a maximal solution to the problem. Moreover, either this solution is globally defined or it blows  up in finite time.

From the previous arguments and estimates, it is standard to obtain the property of persistence of regularity, i.e. if the initial data belongs to some $H^s$, for $s> \frac \alpha 2$, then the maximal solution $u(t)$ of the equation belongs to $H^s$ as long at it exists in $H^{\frac \alpha 2}$. In particular, by density arguments and continuous dependence upon the initial data, we can approximate any $H^{\frac \alpha 2}$ by smooth solutions in $C([0,T],H^{\frac \alpha 2})$, which allows us to prove rigorously the conservation of mass and energy \eqref{inv1} and \eqref{inv2}.
\end{remark}

\subsection{Weak continuity of the flow}
\begin{theorem}\label{TH3}
Let $1< \alpha \leq 2$.
Let $\{u_n\}_n$ be a sequence of $H^{\frac \alpha 2}$ solutions of \eqref{eq:dgBOpm} in $[0,T]$; assume that $u_n(0)\rightharpoonup u_0$ in $H^{\frac \alpha 2}$ weak.
Assume also that (without loss of generality) $\|u_n(0)\|_{H^{\frac \alpha 2}}\leq A$,
$\|u_0\|_{H^{\frac \alpha 2}}\leq A$, $T\leq T(A)$ as in Theorem \ref{TH1}. Then, if $u(t)$ is the solution of \eqref{eq:dgBOpm} corresponding to $u_0$, we have
$$\forall t\in [0,T],\quad 
u_n(t)\rightharpoonup u(t) \quad \text{in $H^{\frac \alpha 2}$ weak.}
$$
\end{theorem}
Note that for $\alpha=1$, the result is proved in the final remark of \cite{CK} (see also \cite{GM}) and for $\alpha =2$, it was proved by different arguments in \cite{MMjmpa}.

\begin{proof}
For $1<\alpha\leq 2$ we remark that  a slight modification of the proof of Theorem \ref{TH1} gives us the local well-posedness in $H^{\frac {\alpha'}2}$ for $1<\alpha'<\alpha$.
Then, the proof is identical to the one in Theorem 5 of \cite{KM}, using this remark.
\end{proof}

\section{Properties of the ground states and perturbation arguments}
In this section, we first recall or prove general results about ground states for \eqref{dgBO} for all $1\leq \alpha \leq 2$, mainly by classical variational arguments. Then, we prove specific results for $\alpha$ close to $2$ by perturbation of the well-known results for gKdV.

\subsection{Existence and first properties of the ground states}
\begin{proposition}\label{le:Q}
Let $1\leq \alpha \leq 2$. There exists a solution $Q\in H^{\frac \alpha 2}(\RR) \cap C^{\infty}(\RR)$ of \eqref{eq:Q} which satisfies the following properties
\begin{itemize}
\item[{\rm (i)}] First properties:
$Q>0$ on $\RR$, $Q$ is even,  $Q'<0$ on $(0,+\infty)$.
\item[{\rm (ii)}] Variational properties. The infima
\begin{equation}\label{eq:mini1}
J_1 = \inf\left\{ \frac {\left(\int | D^{\frac \alpha 2} v |^2\right)\left(\int v^2\right)^{\alpha}}
{\int |v|^{2\alpha +2}},\text{ for } v\in H^{\frac \alpha 2}\right\},
\end{equation}
\begin{equation}\label{eq:mini2}
J_2 = \inf\left\{E(v),\text{ for } v\in H^{\frac \alpha 2} \text{ such that } \int v^2=\int Q^2\right\},
\end{equation}
are attained at $Q$ ($J_2=0$).
\item[{\rm (iii)}] Linearized operator: let $L$ be the unbounded operator defined on $L^2(\RR)$ by 
$$
L v = \DD^\alpha v + v -  Q^{2\alpha} v.
$$
Then, $L$ has only one negative eigenvalue $\mu_0$, associated to an even eigenfunction $\chi_0>0$, $LQ'=0$ and the continuous spectrum of $L$ is $[1,+\infty)$. Moreover, the following holds
\begin{equation}\label{eq:min}
\inf\left\{(L\eta,\eta),\text{ for } \eta\in H^{\frac \alpha 2} \text{ such that } \int \eta Q = 0\right\}=0.
\end{equation}
Finally, let $Q_\lambda(x) = \lambda^{-\frac 1 \alpha} Q(\lambda^{-\frac 2 \alpha} x)$ for all $\lambda>0$  and
\begin{equation}\label{eq:scal}
\Lambda Q  = - \left(\frac d{d\lambda} Q_\lambda\right)_{\lambda = 1 } =
\frac 1\alpha \left( Q + 2 x Q'\right)\quad \text{then}\quad
L ( \Lambda Q ) = - 2 Q.
\end{equation}
\item[{\rm (iv)}] Decay properties:
\begin{equation}\label{eq:decay}
	\forall x\in \RR,\quad Q(x) +(1+|x|) |Q'(x)|+(1+|x|^2)|Q''(x)|\leq \frac C{(1+x^2)^{\frac 12 (1+\alpha)}}.
\end{equation}
\end{itemize}
\end{proposition}
\begin{remark}
In the following we call ground state an even positive solution of  \eqref{eq:Q} in the sense of Proposition \ref{le:Q}. 

The uniqueness (up to translations) of   a positive solution $Q$ of \eqref{eq:Q}  is an open question. 
Even weaker versions of the uniqueness property are open : uniqueness of a ground state, or uniqueness in some neighborhood of $Q$.  
A related open question concerns the kernel of $L$, see below.
\end{remark}

Before proving the above proposition, we recall the following classical result.

\begin{lemma}\label{le:GK} Let $1\leq \alpha \leq 2$.
Let $K(x)$ be such that $\hat K(\xi)=e^{-|\xi|^\alpha}$. Then, $K$ is a real and even function, $K>0$ on $\RR$ and $K'(x)<0$ for $x>0$.
\end{lemma}
\begin{proof}
For $\alpha=1,2$, $K(x)$ is known explicitly.
This result is not trivial for $1<\alpha<2$ but known in probabilistic literature: $K$ is the law of stable distribution, special cases of distribution of class $L$ (see Gnedenko-Kolmogorov \cite{GK54} Theorem page 164). Yamazato \cite{Y78} proved the unimodality of distribution of class $L$, i.e. $K'(x)<0$ for $x>0$.
\end{proof}

\begin{remark}\label{rk:tnt}
It follows in particular from the previous lemma that the operator $|D|^{\alpha}$ for $1\leq \alpha \leq 2$ satisfies properties 
(L1)$_{\alpha/2}$, (L2) and (L3) of \cite{We87}.
\end{remark}
 
We also recall  the following identities satisfied by any solution of \eqref{eq:Q}.

\begin{lemma}\label{LE:app}
Let $Q\in H^{\frac \alpha 2}$ be a solution of \eqref{eq:Q}.
Then,
\begin{equation}\label{eq:app1}
\int Q^2 = \alpha \int |D^{\frac \alpha 2} Q|^2 = \frac {\alpha}{(2\alpha +1)(\alpha +1)} \int Q^{2 \alpha +2}.
\end{equation}
In particular,
$$
E(Q) = \int |D^{\frac \alpha 2} Q|^2 - \frac 1{(2\alpha +1)(\alpha +1)} \int Q^{2 \alpha +2}=0.
$$
\end{lemma}
\begin{proof}
Multiplying equation \eqref{eq:Q} by $Q$ and integrating, we first find
\begin{equation}\label{eq:f1}
	\int |D^{\frac \alpha 2} Q|^2 + \int Q^2 = \frac 1{2 \alpha +1} \int Q^{2 \alpha +2}.
\end{equation}

Second, note that  by Plancherel and integration by parts, for all $u\in \mathcal{S}$, one has
\begin{equation*}
	\int (-|D|^{\alpha} u) (x u_x)   
	= -\frac {\alpha -1} 2 \int |D^{\frac \alpha 2} u|^2. 
\end{equation*}
Thus, multiplying the equation of $Q$ by $x Q'$ and integrating, we obtain
\begin{equation}\label{eq:f2}
	(\alpha-1) \int |D^{\frac \alpha 2} Q|^2 - \int Q^2 = - \frac 1{(2\alpha +1)(\alpha +1)} \int Q^{2 \alpha +2}.
\end{equation}
Combining \eqref{eq:f1} and \eqref{eq:f2}, we find \eqref{eq:app1}.
\end{proof}

\begin{proof}[Sketch of the proof of Proposition \ref{le:Q}]
The existence of a solution $Q$ of \eqref{eq:Q} satisfying (i), (ii) and (iii) follow from Weinstein's arguments \cite{We85, We86, We87, We87b} and Lemma \ref{le:GK}. Property (iv) follows from Amick and Toland's arguments, see \cite{AT}.
 
Let us sketch the proofs.
(i): Follows by Theorem 3.2 in \cite{We87} and  remark \ref{rk:tnt}.

(ii): As in \cite{We85, We87b}, a suitable solution $Q(x)$ is obtained by minimizing the functional $j_1(v)$, defined for $v\in H^{\frac \alpha 2}$ by
$$
	j_1(v)= \frac {\left(\int | D^{\frac \alpha 2} v |^2\right)\left(\int v^2\right)^{\alpha}}
{\int |v|^{2\alpha +2}}.
$$
Note that by Theorem XIII.50  in \cite{RS}, Lemma \ref{le:GK}, Remark \ref{rk:tnt}, and Lemma 6 in \cite{We87}, for all $v\in H^{\frac \alpha 2}$,
$$
(|D|^{\alpha} |v|^*,|v|^*) \leq (|D|^{\alpha} v,v),
$$
where $v^*$ the symmetric decreasing rearrangement of $v$.
Thus, in the minimization procedure, one can always assume that the minimization sequence is composed of nonnegative and even functions.
It is not possible here to use the decay properties of $H^1$ radial functions as in \cite{We85}, since such an argument is limited to space dimensions larger than or equal to $2$ . One rather uses the concentration--compactness approach (\cite{PLL}) on a suitable continuous family of variational problems related to $j_1(v)$, as in \cite{We87}.

Once a nonnegative, symmetric decreasing, minimizer $\psi$ of $j_1$ is constructed, we verify that for some constants $a,b>0$, 
$Q(x) = a \psi(bx)$ satisfies 
$$
\DD^\alpha Q + Q - \tfrac 1{2\alpha+1} {Q^{2\alpha+1}} =0, \quad
Q\in H^{\frac \alpha 2},\quad Q>0,\quad Q'<0 \text{ on $(0,+\infty)$},\quad  Q \text{ even},
$$
and $j_1(Q)=\inf\{j_1(v) \ \text{for } v\in H^{\frac \alpha 2}\}$. By Lemma \ref{LE:app}, we have $E(Q)=0$. In particular, the definition of $J_1$ implies that for all $v\in H^{\frac \alpha 2}$,
\begin{equation}\label{eq:bGN}
	\frac 1{(2\alpha+1)(\alpha+1)} \int |v|^{2\alpha +2} \leq \left(\frac {\int v^2}{\int Q^2}\right)^{\alpha}
	\int |D^{\frac \alpha 2}  v|^2,
\end{equation}
which is the sharp Gagliardo-Nirenberg inequality in this context, and 
which means that if $\int v^2 \leq \int Q^2$, then $E(v)\geq 0$.

Note that also that for two different solutions $Q$ and $\tilde Q$ of \eqref{eq:Q}, both minimizers of $j_1$, we have $\|Q\|_{L^2}= \|\tilde Q\|_{L^2}$.
\medskip

(iii): Exactly as in the proofs of Propositions 4.1 and 4.2 of \cite{We87} and Proposition 2.7 of \cite{We85}, 
we obtain that
$$
0 = \inf\{ (Lv|v), \ \text{for } v \in H^{\frac \alpha 2}, \ \int  v Q=0\},
$$
and $(LQ,Q)<0$, so that there exists exactly one negative eigenvalue $\mu_0$ of $L$, related an even eigenfunction $\chi_0$ which can be taken to be positive. Moreover, the continuous spectrum of $L$ is $[1,+\infty)$.

Finally, from the equation of $Q_\lambda(x+x_0)= \lambda^{-\frac 1\alpha} Q(\lambda^{-\frac 2 \alpha} (x+x_0))$, we have
$$
|D|^{\alpha} Q_\lambda(x+x_0) +  \lambda^2  Q_{\lambda}  (x+x_0) = 
\frac 1{2 \alpha +1}  Q_{\lambda}^{2\alpha +1} (x+x_0).
$$
Differentiating with respect to $x_0$ and taking $x_0=0$, $\lambda=1$, we find $L Q'=0$; differentiating with respect to $\lambda$, and taking $x_0=0$, $\lambda=1$, we find $L (\Lambda Q) = - 2 Q.$

\medskip

(iv): Proof of the decay property.
For this part, we first recall the following facts from \cite{CS}. For a function $F:\RR\to \RR$, we denote by 
$f:\RR_+^2 \to \RR$ ($\RR_+^2=\RR \times [0, +\infty)$) the extension
$$
	f(x,0)=F(x) \ \text{on $\RR$},\quad
	\partial_x^2 f + \partial_y^2 f + \frac {1-\alpha} y \partial_y f = 0 \ \text{on $\RR_+^2$}.
$$
Then, from \cite{CS}, there exists a constant $C_\alpha>0$ such that, on $\RR$,
$$
C_\alpha |D|^{\alpha} F = - \lim_{y\to 0^+} y^{1-\alpha} \partial_y f.
$$
This generalizes a classical observation for $\alpha=1$. 

Next, following \cite{AT,AT2}, if $Q$ is solution of \eqref{eq:Q}, and $q(x,y)$ is its extension to $\RR_+^2$, then $q$ satisfies
\begin{align*}
	& \partial_x^2 q + \partial_y^2 q + \frac {1-\alpha} y \partial_y q = 0 \ \text{on $\RR_+^2$},¬†\\
	& \lim_{y\to 0^+} y^{1-\alpha} \partial_y q = C_\alpha \left(q - \frac 1{2 \alpha +1} q^{2\alpha+1}\right) \ \text{on $y=0$},\\
	& \lim_{|x|\to +\infty} |q(x,0)|=0.
\end{align*}
From \cite{CS} and \cite{AT,AT2}, we are led to set
$$
		G_\alpha (x,y) = \left(\int \frac {dx'}{(1+(x')^2)^{\frac {1+\alpha} 2}} \right)^{-1}
		e^{\frac 1{\alpha C_\alpha} y^\alpha} \int_0^{+\infty}
		e^{-\frac 1{\alpha C_\alpha} (y+\omega)^\alpha} 
		\frac {(y+\omega)^\alpha} {(x^2 + (y+\omega)^2)^{\frac {1+\alpha}2}} d\omega,
$$
so that $q(x,y)$ satisfies on $\RR^2_+$
$$
	q(x,y) = \frac 1{2\alpha+1} \int_{-\infty}^{+\infty} G_\alpha(x-z,y) q^{2\alpha+1} (z,y) dz.
$$
From this expression, we get the decay estimate \eqref{eq:decay} following exactly the same arguments as in pp. 23--24 of \cite{AT2} and using immediate estimates on $G_\alpha$.
\end{proof}

\subsection{Linear Liouville property by perturbation around the gKdV case}
 
We have summarized in Proposition \ref{le:Q} standard results about ground states of \eqref{eq:Q} which hold for any $1\leq \alpha \leq  2$.
It seems that  the following two natural questions are open.
 
\begin{definition}[Uniqueness of the ground state]
 We say that the ground state satisfies the \emph{uniqueness property} if there exists a unique ground state solution of \eqref{eq:Q}.
\end{definition}  

\begin{definition}[Kernel property]\label{cj:kernel} 
Given  a ground state  solution $Q$  of \eqref{eq:Q},
we say that the operator $L$   defined in Proposition \ref{le:Q} satisfies the \emph{kernel property} if
$$
{\rm Ker} (L) = {\rm span}\{Q'\}.
$$
\end{definition}

Note that in the (gKdV) case ($\alpha=2$) by classical ODE arguments, the ground state satisfies the uniqueness property and $L$ satisfies the kernel property 
As a consequence of perturbation arguments, we are able to    prove that, for $\alpha<2$  sufficiently close to $2$, the ground state $Q=Q_{[\alpha]}$ satisfies the uniqueness property and $L=L_{[\alpha]}$ satisfies the kernel property  . See Proposition~\ref{pr:close2} below.  

\medskip

It is standard to observe the following consequence of the kernel property.

 \begin{lemma}
Assume $L$ satisfies  the kernel property. Then, for some constant $\mu>0$,
\begin{equation}\label{eq:kernel}
\forall v \in H^{\frac \alpha 2},  \quad
\int v \chi_0=\int v Q'=0\quad \Rightarrow \quad 
(L v,v) \geq \mu \|v\|_{H^1}^2.
\end{equation}
\end{lemma} 
\begin{proof}
This is a direct consequence of the spectral theorem
and Proposition \ref{le:Q} (iii). Since the operator $L$ has only one negative simple eigenvalue, finitely many positive eigenvalues in $(0,\frac 12]$ (by the Fredholm alternative), and since the eigenvalue $0$ is supposed to be simple, orthogonality with respect to $\chi_0$ and $Q'$ indeed ensures the coercivity of $L$.
\end{proof}   

To study the nonlinear flow around the solitons,  
we will also need the following fundamental rigidity property of the linearized flow around a ground state.

\begin{definition}[Linear Liouville Property]\label{cj:liouville}
We say that $L$ satisfies the \emph{linear Liouville property} if   all
 $H^{\frac \alpha 2}$ bounded   solution $w(t)$   of   $$w_t = \partial_x (L w) \quad 
 (t,x)\in \RR, $$
 such that 
\begin{equation}\label{l2loc} \forall \epsilon>0, \ \exists B>0,\ \forall t\in \RR,   \quad
\int_{|x|> B} |w(t,x)|^2 dx \leq \epsilon
\end{equation}
is necessarily $w(t,x)\equiv c_0 Q'(x)$ for some $c_0\in \RR$. 
\end{definition}  

Clearly,  the linear Liouville property implies  the kernel property, since an element of the kernel of $L$ satisfies the desired conditions. But we do not know if the converse is true, i.e. we do not have a proof of the linear Liouville property assuming the kernel property. It is of course a much stronger property, related to the evolution problem. It was  proved for $\alpha=2$ in \cite{MMjmpa} and \cite{yvanSIAM} by Virial type identities and the variational characterization of $Q$.
Again, we are able to use perturbative arguments to prove this property for $\alpha<2$ sufficiently close to~$2$. 

\medskip
 
We gather these perturbative results in the following proposition. 
 
\begin{proposition}\label{pr:close2}
There exists $\alpha_0\in [1,2)$ such that for all $\alpha_0\leq \alpha\leq 2$, the following properties hold.
\begin{itemize}
\item[{\rm (i)}] There exists a unique (positive, even) ground state solution $Q=Q_{[\alpha]}\in H^1$ of \eqref{eq:Q} and 
$$
Q_{[\alpha]}\to Q_{[2]} \quad \text{as $\alpha \to 2^-$ in $H^1$.}
$$
\item[{\rm (ii)}] Variational characterization of $Q$: $\forall u \in H^{\frac \alpha 2},$
\begin{equation}\label{varchar}
	E(u)=0,\  \int u^2 = \int Q^2,\ \int |\DD^{\frac \alpha 2}u |^2
	= \int |\DD^{\frac \alpha 2} Q |^2 \ \Rightarrow \ u = \pm Q(.-x_0), \ x_0 \in \RR.
\end{equation}
\item[{\rm (iii)}] The kernel property holds true.
\item[{\rm (iv)}] The linear Liouville property holds true.
\end{itemize}
\end{proposition}
\begin{proof}[Proof of Proposition \ref{pr:close2}]
The proof of Proposition \ref{pr:close2} is perturbative. Let us denote by $Q_{[2]}$ the unique positive even solution of \eqref{eq:Q} given by \eqref{explicitQ}.

\medskip

(i) Let  $\alpha_n \to 2$ be an increasing sequence and for all $n$, let $Q_{[\alpha_n]}$ be a solution of \eqref{eq:Q} given by Proposition \ref{le:Q}. First, we claim that
$ \lim_{n\to +\infty} Q_{[\alpha_n]} = Q_{[2]}$.
Indeed, from \eqref{eq:bGN} applied to a given function $w$, we obtain $\int Q^2_{[\alpha]} \leq C$.
Then, by Lemma~\ref{LE:app},  
$\|Q_{[\alpha_n]}\|_{H^{\frac {\alpha_n} 2}} \leq C$,
and using the equation of $Q_{[\alpha_n]}$, it follows that $Q_{[\alpha_n]}\in H^1$ and
$\|Q_{[\alpha_n]}\|_{H^1} \leq C $.
In particular, there exists $V\in H^1$, a weak limit in $H^1$ of a subsequence of $Q_{[\alpha_n]}$, still denoted by $Q_{[\alpha_n]}$. It is easy to see that $V \neq 0$, using  Lemma \ref{LE:app}.
Indeed, since
$$
\int Q_{[\alpha_n]}^2 \leq C \|Q_{[\alpha_n]}\|_{L^\infty}^{2\alpha} \int Q_{[\alpha_n]}^2,
$$
it follows that $Q_{[\alpha_n]}(0)=\|Q_{[\alpha_n]}\|_{L^\infty}\geq c_1>0$  and since weak $H^1$ convergence implies uniform convergence on compact sets, we obtain $V(0)\neq 0$.

 Moreover, we easily check that $V$ satisfies
equation \eqref{eq:Q} with $\alpha=2$ and thus by uniqueness, we deduce $V=Q_{[2]}$.

To obtain the strong convergence, we just observe that 
$$
\limsup_{n\to +\infty} \int Q_{[\alpha_n]}^2 \leq \int Q_{[2]}^2
$$
follows from the following consequence of Lemma \ref{LE:app}
\begin{align*}
[(\alpha_n+1)(2\alpha_n+1)]^{-1} \left(\int Q^2_{[\alpha_n]} \right)^{\alpha_n}
& =j_{1,[\alpha_n]}(Q_{[\alpha_n]})\\&  \leq j_{1,[\alpha_n]}(Q_{[2]}) \to j_{1,[2]}(Q_{[2]}) 
= [15]^{-1}  \left(\int Q^2_{[2]} \right)^2.
\end{align*}
This gives $L^2$ strong convergence. To obtain $H^1$ convergence, we just use the equation of $Q_{[\alpha_n]}$ and interpolation argument.

Second, we consider two sequences $Q_{[\alpha_n]}$ and $\widetilde Q_{[\alpha_n]}$ of solutions of \eqref{eq:Q} as in Proposition \ref{le:Q}.
By the first observation, we have
$Q_{[\alpha_n]}\to Q_{[2]}$ and $\widetilde Q_{[\alpha_n]}\to Q_{[2]}$ in $H^1(\RR)$.
Moreover, by the equation satisfied by $Q_{[\alpha_n]}$ and $\widetilde Q_{[\alpha_n]}$, we have
\begin{equation}\label{eq:bd}
	\|\DD^{\alpha_n}(Q_{[\alpha_n]}-\widetilde Q_{[\alpha_n]})\|_{L^2}
	\leq C \| Q_{[\alpha_n]}-\widetilde Q_{[\alpha_n]} \|_{L^2}.
\end{equation} 

Let
$$
	w_n = \frac {Q_{[\alpha_n]} - \widetilde Q_{[\alpha_n]}} {\|Q_{[\alpha_n]} - \widetilde Q_{[\alpha_n]}\|_{H^1}.}
$$
By \eqref{eq:bd}, the sequence $w_n$ is bounded in $H^{\frac 32}$ (say $\alpha_n >3/2)$.
A more precise computation using the equations of $Q_{[\alpha_n]}$ and $\widetilde Q_{[\alpha_n]}$ shows that the function $w_n$ satisfies
$$
	\|L_{[\alpha_n]} w_n\|_{H^1} = \||D|^{\alpha_n} w_n + w_n - Q_{[\alpha_n]}^{2\alpha_n} w_n\|_{H^1}
	\le C  \| Q_{[\alpha_n]}-\widetilde Q_{[\alpha_n]} \|_{L^2} $$
where we observe a special cancellation.	
Using this estimate, the bound of the sequence $(w_n)$ in $H^{\frac 32}$ and standard Fourier analysis, we find
$$
	\lim_{n\to +\infty} (L_{[2]} w_n,w_n)_{L^2} = 0.
$$
It is known that \eqref{eq:kernel} holds for $\alpha=2$,
moreover, it can be rewritten as
$$
	\forall v \in H^{\frac \alpha 2}, \quad
	(L_{[2]} v , v ) \geq  \frac \mu 2 \|v\|_{H^1}^2 - C \left(\int v \chi_0\right)^2 - C \left(\int v Q'\right)^2.
$$

By parity properties, we observe $\int w_n Q'_{[2]}=0$.
By the previous equation, and  \eqref{eq:kernel},
we have 
$$
	\int w_n \chi_{0,[2]} = \frac 1{\mu_0} (L_{[2]} \chi_0 , w_n )
	=  \frac 1{\mu_0} (L_{[\alpha_n]} \chi_0 , w_n ) + o(1) 
	=  \frac 1{\mu_0} ( \chi_0 ,L_{[\alpha_n]} w_n ) + o(1),
$$
and thus $\lim_{n\to +\infty}  \int w_n \chi_{0,[2]}= 0$.
Since $\|w_n\|_{H^1}=1$, we find a contradiction for $n$ large enough.

Therefore, there exists $\alpha_0\in [1,2)$ so that there is one and only one solution of \eqref{eq:Q} satisfying the properties of Proposition \ref{le:Q}.

\medskip

(ii) Variational characterization. It follows from the arguments of the proof of Proposition~\ref{le:Q}. Indeed, for such a function $u$, $|u|$ is a minimizer of $J_1$ and satisfies the same equation as $Q$.
By the uniqueness result of (i), it follows that $|u|$ is a translation of $Q$. Thus, $u$ being continuous, it is a translation of $Q$ or $-Q$.

\medskip

(iii) Using a similar argument and possibly taking  $\alpha_0$ closer to $2$, we can prove directly that 
${\rm Ker}(L_{[\alpha]}) = {\rm span} \{ Q'_{[\alpha]} \}$ for $\alpha \in [\alpha_0,2]$.
It is also a consequence of the linear Liouville property proved below.

\medskip

(iv)
Now, we prove the linear Liouville property for $\alpha$ close to $2$.
The proof is by contradiction and similar to (i), using a compactness argument.
For the sake of contradiction, we assume that there exists an increasing sequence $\alpha_{n}\to 2$ and functions $w_{n}(t,x)$ satisfying
\begin{align*}
	& (w_{n})_{t} = (L_{[\alpha_n]} w_{n})_{x},\\
	& w_{n}(t) \not \equiv a_{n}(t)Q_{[\alpha_{n}]}',\quad \sup_{{t\in \RR}}\|w_{n}(t)\|_{H^{\frac {\alpha_{n}} 2}}\leq C_{n},\\
	& \forall \epsilon>0,\exists B_{n}(\epsilon)>0,\forall t\in \RR,\quad
	\int_{|x|>B_{n}(\epsilon)} |w_{n}(t,x)|^2 dx \leq \epsilon.
\end{align*}
We introduce several auxiliary functions defined from $w_{n}$.
First, set
$$
	\tilde w_{n}(t) = w_{n}(t) - \frac {\int Q_{[\alpha_{n}]}' w_{n}(t)}{\int (Q_{[\alpha_{n}]}')^2} Q_{[\alpha_{n}]}',
$$ satisfying
\begin{align*}
	& (\tilde w_{n})_{t} = (L_{[\alpha_n]}\tilde  w_{n})_{x}+ \delta_{n}(t)Q_{[\alpha_{n}]}',\\
	& \tilde w_{n}(t) \not \equiv0,\quad \sup_{{t\in \RR}}\|\tilde w_{n}(t)\|_{H^{\frac {\alpha_{n}} 2}}\leq C_{n}',\quad 
	\int \tilde w_{n}(t) Q_{[\alpha_{n}]}' =0,\\
	& \forall \epsilon>0,\exists B_{n}(\epsilon)>0,\forall t\in \RR,\quad
	\int_{|x|>B_{n}(\epsilon)} |\tilde w_{n}(t,x)|^2 dx \leq \epsilon.
\end{align*}
Moreover, using monotonicity arguments on $\tilde w_n(t)$ as in Section 4 of the present paper and Lemma 4 in \cite{yvanSIAM},
we find ($\alpha_{n}>3/2$)
$$
\forall x_{0}>1,\forall t\in \RR,\quad \int_{|x|>x_{0}} |\tilde w_{n}(t,x)|^2 dx \leq
\sup_{t\in \RR} \|\tilde w_{n}(t)\|_{L^2}^2 \frac C {|x_{0}|^{\frac 32}}.
$$
In particular, by Fubini, we obtain
$$
\forall t\in \RR,\quad \int |x| |\tilde w_{n} (t)|^2 \leq  C \sup_{t\in \RR} \|\tilde w_{n}(t)\|_{L^2}^2.
$$
Multiplying the equation of $\tilde w_n$ by $x \tilde w_n$ and using the argument of Lemma \ref{LE:app},
we find, for $C>0$,
$$
\frac d{dt} \int x (\tilde w_n(t))^2 \leq - C \| |D|^{\frac \alpha 2} \tilde w_{n}(t)\|_{L^2}^2
+ C' \|\tilde w_{n}(t)\|_{L^2}^2,
$$
and thus, for all $t\in \RR$,
$
\int_t^{t+1} \| |D|^{\frac \alpha 2} \tilde w_{n}(t)\|_{L^2}^2 \leq C \sup_{t\in \RR} \|\tilde w_{n}(t)\|_{L^2}^2.
$
Therefore, from  standard arguments, using the equation of $\tilde w_n$,
$$
\sup_{t\in \RR} \|\tilde w_{n}(t)\|_{H^{\frac {\alpha_{n}} 2}} \leq
 C \sup_{t\in \RR} \|\tilde w_{n}(t)\|_{L^2},
$$
for a constant $C>0$ independent of $n$.

Let $t_{n}$ be such that
$
\|\tilde w_{n}(t_{n})\|_{L^2}\geq \frac 12 \sup_{t\in \RR} \|\tilde w_{n}(t)\|_{L^2}
$
and set
$$
\bar w_{n}(t,x) = \frac {\tilde w_{n} (t_{n}+t,x)}{\sup_{t\in \RR} \|\tilde w_{n}(t)\|_{L^2}},
$$
so that we have
\begin{align*}
	& (\bar w_{n})_{t} = (L_{[\alpha_n]}\bar  w_{n})_{x}+ \bar\delta_{n}(t)Q_{[\alpha_{n}]}',\\
	& \|\bar w_{n}(0) \|_{L^2}\geq \frac 12,\quad \sup_{{t\in \RR}}\|\bar w_{n}(t)\|_{H^\frac {\alpha_{n}} 2 }\leq C,\quad 
	\int \bar w_{n}(t) Q_{[\alpha_{n}]}' =0,\\
	& \bar \delta_{n}(t) = \frac 1{\int (Q_{[\alpha_n]}')^2} \int \bar w_{n} L_{[\alpha_{n}]}(Q_{[\alpha_{n}]}''),\\
	& \forall x_{0}>1,\forall t\in \RR,\quad \int_{|x|>x_{0}} |\bar w_{n}(t,x)|^2 dx \leq
\frac C {|x_{0}|^{\frac 32}}.
\end{align*}

Finally, we set
$$
	\hat w_{n}(t) = \bar w_{n}(t) - Q_{[\alpha_{n}]}' \int_{0}^t \bar \delta_{n}(s) ds,
$$
so that
\begin{align*}
	& (\hat w_{n})_{t} = (L_{[\alpha_n]}\hat  w_{n})_{x},\\
	& \|\hat w_{n}(0) \|_{L^2}\geq \frac 12,\quad \|\hat w_{n}(0)\|_{H^{\frac {\alpha_{n}}2 }}\leq C,\quad 	\int \hat w_{n}(0) Q_{[\alpha_{n}]}' =0,\\ &
	\forall x_{0}>1, \quad \int_{|x|>x_{0}} |\hat w_{n}(0,x)|^2 dx \leq
\frac C {|x_{0}|^{\frac 32}}.
\end{align*}
We are now able to pass to the strong limit in  $H^{1^-}$, for any $0<1^-<1$. 
$$
	\hat w_{n}(0)\to \hat w_{0}\not \equiv 0,
$$
and we define the solution $\hat w(t)$ of 
\begin{align*}
	 (\hat w)_{t} = (L_{[2]}\hat  w)_{x},\quad 
	 \hat w(0)=\hat w_{0}.
\end{align*}
By wellposedness argument in $H^{1^-}$, we have
$
\hat w_{n}(t)\to \hat w(t)$ in $H^{1^-}$.  
Moreover,
$$
\bar \delta_{n} (t) \to \bar \delta (t) = 
\frac 1{\int (Q_{{[2]}}')^2}\int \hat w(t) L_{{[2]}} (Q_{{[2]}}'').
$$
Set
$
	\bar w(t) = \hat w(t) + Q_{[2]}' \int_{0}^t \bar \delta(s) ds.
$
Then
\begin{align*}
&\forall t\in \RR,\quad  \bar w_{n}(t) \to \bar w(t) \textrm{ in } H^{1^-},
\\
	& \bar w_{t} =(L_{[2]} \bar w)_{x} +\bar \delta Q_{{[2]}}',
\\
	& \bar w(0)  \not \equiv0 ,\quad 	\int \bar  w(0) Q_{[2]}' =0,\\ &
	\forall t\in \RR,\ \forall x_{0}>1, \quad \int_{|x|>x_{0}} |\bar w(t,x)|^2 dx \leq
\frac C {|x_{0}|^{\frac 32}}.
\end{align*}
But the existence of such a $\bar w$  is a contradiction with Theorem 1 in \cite{yvanSIAM}, i.e. the linear Liouville property for the gKdV case (see also \cite{MMjmpa}).
\end{proof}

\section{Modulation and monotonicity   for solutions close to solitons}
In this section, we consider $1\leq \alpha \leq 2$ and $Q$ is any ground state solution of \eqref{eq:Q}.
\subsection{Modulation}
\begin{lemma}[Modulation of a solution close to  the family of solitons]\label{MODULATION}
There exist $C,\epsilon_0>0$ such that for any $0<\epsilon<\epsilon_0$,
if $u(t)$ is an $H^\frac \alpha 2$ solution of \eqref{dgBO} such that for $t_1<t_2$ and
$\lambda_0(t)>0$, $\rho_0(t)\in \RR$, defined on $[t_1,t_2]$,
\begin{equation}
\forall t\in [t_1,t_2],\quad  \|u(t)-Q_{\lambda_0(t)}(.-\rho_0(t))\|_{H^{\frac \alpha 2}}
< \epsilon,
\end{equation}
then there exist  $\lambda(t)>0$, $\rho(t)\in C^1([t_1,t_2])$ such that
\begin{equation}\label{eq:35b}
\eta(t,y)=\lambda^{\frac 1\alpha}(t) u\left(t,\lambda^{\frac 2 \alpha}(t) y+\rho(t)\right)-Q(y) 
\end{equation}
satisfies
\begin{align}
\forall  t\in  [t_1,t_2],\quad 
  &  \int Q'(y) \eta(t,y)dy= \int \chi_0(y) \eta(t,y)dy=0,\quad
	  \|\eta(t)\|_{H^{\frac \alpha 2}}\leq C \epsilon, \label{ortho1}
\\
	& \left|\frac {\lambda_0(t)}{\lambda(t)}\right| + \left| \rho_0(t)-\rho(t)\right| \leq C \epsilon.\label{ortho3}
\end{align}

Moreover, setting
$$
	s=\int_0^t \frac {dt'}{\lambda^{2+\frac 2 \alpha}(t')} , \quad
	\Lambda \eta = \frac 1 \alpha ( \eta + 2 y \eta_y),
$$
the function $\eta(s,x)$ is solution of
\begin{align*}
&  \eta_s   - \partial_y(L\eta)    = \frac {\lambda_s}{\lambda} \Lambda Q + \left(\frac {\rho_s}{\lambda^{\frac 2 \alpha}} -1 \right) Q'
 + \frac {\lambda_s}{\lambda} \Lambda \eta + \left(\frac {\rho_s}{\lambda^{\frac 2 \alpha}} -1 \right) \eta_y - \partial_y (\mathcal{R}(\eta)), 
\\
& \hbox{where} \quad \mathcal{R}(\eta)= \frac 1{2 \alpha +1} |Q+\eta|^{2\alpha} (Q+\eta) - \frac 1{2 \alpha +1}  Q^{2\alpha +1} - Q^{2 \alpha} \eta,
\end{align*}
and the following holds
\begin{equation}\label{ortho2}
  	\left|\frac {\rho_s(s)}{\lambda^{\frac 2 \alpha}(s)}-1\right|+
	\bigg| \frac {\lambda_s(s)}{\lambda(s)}\bigg|\leq  C \left(\int \frac {\eta^2(s,y)}{1+y^2}dy\right)^{\frac 12}\leq  C\|\eta(s)\|_{L^2}.
\end{equation}
\end{lemma}
\begin{proof}[Sketch of proof of Lemma \ref{MODULATION}.]
This result is completely proved for $\alpha=2$ in \cite{MMgafa}.
For $1\leq \alpha < 2$, the proof is exactly the same. In particular, the existence of the modulation parameters 
$(\lambda(t),\rho(t))$ such that \eqref{ortho1} hold is based on the implicit function theorem.

Then, the equation of $\eta(t)$, $\lambda(t)$ and $\rho(t)$ is easily obtained from the equation of $u(t)$,
and the estimates \eqref{ortho2} on $\lambda_s$, $\rho_s$ follow from the equation of $\eta$ multiplied by $\chi_0$ and $Q'$.
Indeed, let us first introduce
$$v(t,y)=\lambda^{\frac 1 \alpha}(t) u(t, \lambda^{\frac 2 \alpha} y + \rho(t)).$$
Then, $v(t,y)$ satisfies
\begin{align*}
	\lambda^{\frac {2 \alpha +2 }{\alpha}} v_t - \partial_y (|D|^\alpha v) + |v|^{2 \alpha} \partial_y v
	- \lambda^{\frac {2 \alpha + 2}{\alpha}} \frac {\lambda_t}{\lambda}  \Lambda v 
- \lambda^{\frac {2 \alpha +2 }{\alpha}}  \frac {\rho_t}{\lambda^{\frac 2 \alpha}} \partial_y v=0.
\end{align*}
Using the new time variable $s$, since ${\lambda^{\frac {2 \alpha+2}{\alpha}}} ds= dt$,
$$
v_s - \partial_y \left(|D|^\alpha v + v - \frac 1{1+2\alpha} |v|^{2 \alpha} v\right)
	=  \frac {\lambda_s}{\lambda}  \Lambda v 
 + \left(\frac {\rho_s}{\lambda^{\frac 2 \alpha}}-1\right) \partial_y v.
$$
Now, expanding $v=Q+\eta$ and using the equation of $Q$, we find
\begin{align*}
 \eta_s - \partial_y(L\eta)   &= \frac {\lambda_s}{\lambda} \Lambda Q + \left(\frac {\rho_s}{\lambda^{\frac 2 \alpha}} -1 \right) Q'
 + \frac {\lambda_s}{\lambda} \Lambda \eta + \left(\frac {\rho_s}{\lambda^{\frac 2 \alpha}} -1 \right) \eta_y
\\
& - \partial_y \left( \frac 1{2 \alpha +1} |Q+\eta|^{2\alpha} (Q+\eta) - \frac 1{2 \alpha +1}  Q^{2\alpha +1} - Q^{2 \alpha} \eta\right).
\end{align*}
To prove \eqref{ortho2}, we multiply the above equation by $\chi_0$ and then by $Q'$ and we use
the orthogonality conditions \eqref{ortho1}. Indeed, using decay properties of $\chi_0$ and $Q'$
(proved as in Proposition~\ref{le:Q}, iv) and 
$(\Lambda Q, \chi_0) = -\frac 1{\mu_0} (\Lambda Q, L\chi_0) = \frac 2 {\mu_0} (Q,\chi_0) \neq 0$,
$(Q',\chi_0)=0$, $(\Lambda Q,Q')=0$,
 we obtain
 $$
 \left|\frac {\lambda_s}{\lambda} \right|+\left|\frac {\rho_s}{\lambda^{\frac 2 \alpha}} -1 \right|
 \leq 
 C \left(\int \frac {\eta^2 }{1+y^2}dy\right)^{\frac 12}+
 C \left(\left|\frac {\lambda_s}{\lambda} \right| +\left|\frac {\rho_s}{\lambda^{\frac 2 \alpha}} -1 \right|\right) \|\eta\|_{L^2},
 $$
 and for $\epsilon_0$ small enough, we obtain \eqref{ortho2}.
\end{proof}

\subsection{Monotonicity argument on $u(t)$}
This section contains the main new argument of this paper, i.e. the extension to 
equation \eqref{dgBO} of the $L^2$ monotonicity arguments proved in \cite{MMjmpa}, \cite{Me} for the (gKdV) equation and in \cite{KM} for the (BO) equation. With respect to the (gKdV) case, the difficulty comes from the nonlocal character of the operator in \eqref{dgBO}. Note that in \cite{KM}, using  special symmetry  arguments and harmonic extensions, we could overcome the difficulty created the nonlocal operator $|D|$.
For the general case of equation \eqref{dgBO} with $1<\alpha <2$, we can prove similar results using pseudo-differential operators tools. This is the objective of this section.

Using the standard notation $\langle x \rangle^2 = 1+x^2$, we set, for $\frac{1}{2}<r \leq \frac 12(\alpha+1)$ to be chosen later 
$$
\varphi(x)=\int_{-\infty}^x \frac {ds}{\langle s \rangle^{2r}},\quad
\phi(x)=\frac {1}{\langle s \rangle^{r}}=\sqrt{\varphi'}.
$$
For $A>1$ to be chosen, let
$$
	\varphi_A(x) = \varphi\left(\frac x A\right).
$$

We now claim the following $L^2$ monotonicity results.

\begin{proposition}\label{PR:2}
Let $r\in (\frac 12 ,  \frac 12 (\alpha+1)]$ and $0<\mu<1$.
	Under the assumptions of Lemma \ref{MODULATION}, 
	assuming in addition
	\begin{equation}\label{eq:surL}
	\forall t\in [t_1,t_2],\quad
	\lambda(t) \leq 2.
	\end{equation}
	for $\epsilon_0=\epsilon_0(\mu,r)$ small enough and $A=A(\mu,r)$ large enough, there exists $C_0=C(\mu,r,A)>0$ such that for all $x_0>1$, 
	\begin{enumerate}
	\item[{\rm (i)}] Monotonicity on the right of the soliton: 
	\begin{equation}\label{monotonicity1}\begin{split}
	&\int u^2(t_2,x)\varphi_A(x-\rho(t_2)-x_0) dx 
	 \\&  \leq \int u^2(t_1,x)\varphi_A(x-\rho(t_1)-\mu (\rho(t_2)-\rho(t_1))-x_0) dx +\frac {C_0} {x_0^{2r-1}}.
	\end{split}\end{equation}
	\item[{\rm (ii)}] Monotonicity on the left of the soliton:  \begin{equation}\label{monotonicity2}
	\begin{split}
	&\int u^2(t_2,x)\varphi_A(x-\rho(t_2)+\mu (\rho(t_2)-\rho(t_1))+x_0) dx 
	 \\&  \leq \int u^2(t_1,x)\varphi_A(x-\rho(t_1)+x_0) dx +\frac {C_0} {x_0^{2r-1}}.
	\end{split}\end{equation}
	\end{enumerate}
\end{proposition}

The case $\alpha=1$ is treated in \cite{KM} by different techniques. 
For $\alpha =2$, the error term is in fact exponential in $x_0$. See e.g. \cite{MMjmpa}.
 
\begin{proof}
Let $u(t)$ be a solution of \eqref{dgBO} under the assumptions of Lemma  \ref{MODULATION}.
By standard regularization arguments (density arguments and continuous dependence of the solution of \eqref{dgBO} upon the initial data), we may assume that $u(t)$ is smooth (see Remark \ref{RE:bip}).
We prove \eqref{monotonicity1}. 
Estimate \eqref{monotonicity2} is then deduced from \eqref{monotonicity1}, $L^2$-norm conservation and the symmetry $x\to -x,$ $t\to -t$ of the equation.

For $0<\mu< 1$, $x_0>1$ and   any
$t\in [t_1,t_2]$, $x\in \RR$, set
$$
\tilde x = x-x_0 -\rho(t) - \mu ( \rho(t_2)  - \rho(t)), \quad
M_\varphi(t)=M_{\varphi,A,x_0,t_2}(t) = \frac 12 \int u^2(t,x) \varphi_A(\tilde x) dx.
$$
By direct computations, we have the following generalization of the well-known Kato identity 
(\cite{KATO})
\begin{align}
	   \frac{d}{dt} M_\varphi (t)&   = \frac {\mu-1} 2  \rho_t \int u^2 \varphi_A'(\tilde x) dx+ \int u_t u \varphi_A(\tilde x) dx \\ 
	&= \frac { \mu -1}2 \rho_t \int u^2 \varphi_A'(\tilde x) dx -\int (\partial_x (-\DD^\alpha u) + |u|^{2 \alpha} u_x ) u \varphi_A(\tilde x) dx \nonumber \\
	&= \frac {\mu-1} 2  \rho_t \int u^2 \varphi_A'(\tilde x) dx +\int (-\DD^\alpha u)( u_x \varphi_A(\tilde x) + u \varphi_A'(\tilde  x) ) dx \label{eq:KATO}\\ &+\frac 1{2(\alpha + 1)} \int |u|^{2\alpha + 2} 
	 \varphi_A'(\tilde x) dx.\nonumber
\end{align}

Two terms in the right-hand side of \eqref{eq:KATO} are treated by the following two lemmas.

\begin{lemma}\label{lemma2} 
Let $\alpha \in [1,2]$, and $r\in (\frac 12 ,\frac 12 (\alpha+1)]$. There exists $C>0$ such that, 
for all $u\in \SC$,
$$
  \int (-\DD^\alpha u)u_x \varphi(x) \leq
- \frac {(\alpha-1)} 2 \int  \left( \DD^{\frac \alpha 2}(\phi u)\right)^2
 + C \int u^2 \varphi'(x) dx.
$$
\end{lemma} 

\begin{lemma}\label{lemma1}
Let $\alpha \in [1,2]$, and $r\in (\frac 12 ,\frac 12 (\alpha+1)]$.
There exists $C>0$ such that, for all $u\in \SC$,
$$
\int (-\DD^\alpha u)u \varphi'(x) dx \leq 
-\int  \left( \DD^{\frac \alpha 2}(\phi u)\right)^2
 +C \int u^2 \varphi'(x) dx.
$$
\end{lemma}

Assuming Lemmas \ref{lemma2}--\ref{lemma1}, we finish the proof of the proposition.
 First, note that from Lemmas \ref{lemma2} and \ref{lemma1}, by changing variables ($x'=x/A$), we find for any $u\in \SC$,
\begin{align}
&  \int (-\DD^\alpha u)u_x \varphi_A(x) \leq 
- \frac {(\alpha-1)} 2 \int  \left( \DD^{\frac \alpha 2}( u\sqrt{\varphi_A'}\,)\right)^2
 +   \frac C {A^{\alpha}} \int u^2 \varphi_A'(x) dx, \label{eq:41b}\\
&\int (-\DD^\alpha u)u \varphi'_A(x) dx \leq -\int  \left( \DD^{\frac \alpha 2}( u\sqrt{\varphi_A'}\,)\right)^2+\frac C {A^{\alpha}}  \int u^2 \varphi_A'(x) dx.
\label{eq:41c}
\end{align}
By \eqref{eq:KATO}, \eqref{eq:41b}, \eqref{eq:41c}, we find
$$
M_{\varphi}'(t) \leq - \frac 12 \left(\rho_t (1 - \mu) - \frac C {A^\alpha} \right) \int u^2(t) \varphi_A'(\tilde x) dx 
+ \frac 1{2 (\alpha +1)} \int |u|^{2 \alpha +2} \varphi_A'(\tilde x) dx.
$$
Note that from \eqref{ortho2} for  $\epsilon_0$ small enough
$$
\frac 1{\lambda^2}\left|\frac {\rho_s} {\lambda^{\frac 2 \alpha}} -1 \right| 
=  \left|  {\rho_t}  - \frac 1{\lambda^2}\right|
\leq \frac 1{10} \frac 1{\lambda^2}.
$$
In particular, since $\lambda<2$, $\rho_t > 1/5$.
Choosing $A$ large enough, we find
$$
	M_{\varphi}'(t) \leq - \frac {1-\mu} 4 \rho_t   \int u^2(t) \varphi_A'(\tilde x) dx
	+  \frac 1{2 (\alpha +1)} \int |u|^{2 \alpha +2} \varphi_A'(\tilde x) dx.
$$
The constant $A>0$ is now fixed.

Now, we estimate the nonlinear term as in \cite{MMjmpa}, using the decomposition \eqref{eq:35b} and
the decay of $Q$ \eqref{eq:decay}.
Let $a_0$ to be fixed later. 
We decompose the nonlinear term as follows
$$
 \int |u|^{2 \alpha +2} \varphi_A'(\tilde x) dx = \mathbf{I} + \mathbf{II},
$$
where
$$
\mathbf{I} = \int_{|x-\rho(t)|>a_0} |u|^{2 \alpha +2} \varphi_A'(\tilde x) dx
\quad \text{and}\quad
\mathbf{II} = \int_{|x-\rho(t)|<a_0} |u|^{2 \alpha +2} \varphi_A'(\tilde x) dx.
$$
On the one hand 
\begin{align*}
\mathbf{I} & \leq \|u(t)\|_{L^\infty(|x-\rho(t)|>a_0)}^{2 \alpha} \int u^2 \varphi_A'(\tilde x)
\\ & \leq
C \left( \|Q_{\lambda(t)}\|_{L^\infty(|x|>a_0)}^{2 \alpha}
+ \|\lambda^{-\frac 1 \alpha} \eta(t,\lambda^{- \frac 2 \alpha} .)\|_{L^\infty(|x|>a_0)}^{2 \alpha}
\right) \int u^2 \varphi_A'(\tilde x)\\
& \leq C \lambda^{-2}(t) \left( \|Q\|_{L^\infty(|y|\geq 2^{-\frac 2 \alpha} a_0)}^{2 \alpha} + \|\eta\|_{L^\infty}^{2 \alpha}\right) \int u^2 \varphi_A'(\tilde x)\\
& \leq C \rho_t  \left(  \|Q\|_{L^\infty(|y|\geq 2^{-\frac 2 \alpha} a_0)}^{2 \alpha}+ \|\eta\|_{H^{\frac \alpha 2}}^{2 \alpha}\right) \int u^2 \varphi_A'(\tilde x)\leq 
\frac {1-\mu}{8}  \rho_t \int u^2 \varphi_A'(\tilde x),
\end{align*}
for $a_0$ large enough and $\epsilon_0$ small enough (recall that $\lambda(t)\leq 2$, $1\leq \alpha \leq 2$).

On the other hand, $a_0$ being now fixed, by \eqref{eq:35b} and \eqref{ortho2},
$$
\|u(t)\|_{L^\infty}^{2 \alpha} \leq \frac C{\lambda^2(t)} \leq C' \rho_t.
$$
Thus, by the definition of $\varphi_A$
$$
\mathbf{II} \leq \|u(t)\|_{L^2}^2 \|u(t)\|_{L^\infty}^{2 \alpha} \|\varphi_A'(\tilde x)\|_{L^\infty(|x-\rho(t)|<a_0)}\leq C \rho_t \langle x_0 + \mu (\rho(t_2)-\rho(t)) \rangle^{-2r}.
$$

Estimate \eqref{monotonicity1} is thus obtained by
integration on $[t_1,t_2]$. 
\end{proof}

Now, we prove Lemmas \ref{lemma2}--\ref{lemma1}.
\begin{proof}[Proof of Lemma \ref{lemma2}]
We use commutator arguments and pseudo-differential operators tools. We recall here some well-known results which can be found for instance in H\"ormander \cite{Ho} chapter 18. For simplicity we denote by $(u|v)=\int u(x)\overline{v(x)}dx$ and $\|u\|^2=(u|u)$.

We denote by $S^{m,q}$ the symbolic class of symbol defined by

\begin{equation}
 a(x,\xi)\in S^{m,q} \Leftrightarrow
 \left\{
\begin{array}{l}
 a\in{\cal C}^\infty(\RR^2),\\
\forall k, \beta\in\NN ,\ \exists C_{k,\beta}>0 \textrm{ such that }|\partial_x^k\partial_\xi^\beta a(x,\xi)|\le C_{k,\beta}\langle x\rangle^{q-k}\langle\xi \rangle^{m-\beta}
\end{array}
\right.
\end{equation}
Following H\"ormander's notation, we have $S^{m,q}=S(\langle x\rangle^q\langle \xi\rangle^m,g)$ where $g=\frac{dx^2}{\langle x\rangle^2}+\frac{d\xi^2}{\langle \xi\rangle^2}$.
We define the operator associated to $a$ by the following formula for  $u\in {\cal S}$,

\begin{equation}\label{def-PDO}
 a(x,D)u=\frac1{2\pi}\int e^{ix\xi}a(x,\xi)\hat u(\xi)d\xi
\end{equation}
where the Fourier transform is defined by $\hat u(\xi)=\int e^{-ix\xi}u(x)dx$. We recall here some results about the pseudo-differential calculus.

\begin{equation}\label{ineq-L2}
\textrm{Let }a(x,\xi)\in S^{m,q}, \ \exists C>0,\ \forall u\in{\cal S} \textrm{ then }\|a(x,D)u\|\le C\| \langle x\rangle^q\langle D\rangle^mu\|
\end{equation}

\begin{equation}\label{Adjoint}
\begin{array}{l}
\textrm{Let }a(x,\xi)\in S^{m,q}, \textrm{ there exists }b(x,\xi)\in  S^{m,q}\textrm{ such that } a(x,D)^*=b(x,D)\\[2pt]
\textrm{moreover, there exists }r_0(x,\xi)\in S^{m-3,q-3} \textrm{ such that }\\[2pt]
b(x,\xi)=\overline{a(x,\xi)}+\frac1i\partial_x\partial_\xi\overline{a(x,\xi)}-\frac12\partial_x^2\partial_\xi^2\overline{a(x,\xi)}+r_0(x,\xi)
\end{array}
\end{equation}
We recall that $A^*$ is the unique operator satisfying for all $u$ and $v$ in $\cal S$, $(Au|v)=(u|A^*v)$. 
We remark that $\partial_x\partial_\xi\overline{a(x,\xi)}\in S^{m-1,q-1}$ and $\partial_x^2\partial_\xi^2\overline{a(x,\xi)}\in S^{m-2,q-2}$.

\begin{equation}\label{Composition}
\begin{array}{l}
 \textrm{Let }a(x,\xi)\in S^{m,q} \textrm{ and }b(x,\xi)\in S^{m',q'} \textrm{ then there exists }c(x,\xi)\in S^{m+m',q+q'} \\
\textrm{such that } a(x,D)b(x,D)=c(x,D).
\end{array}
\end{equation}
Remark that following \eqref{def-PDO}, we have $a(x,D)D=c(x,D)$ where $c(x,\xi)=a(x,\xi)\xi$.

\begin{equation}\label{Commutator}
\begin{array}{l}
 \textrm{Let }a(x,\xi)\in S^{m,q} \textrm{ and }b(x,\xi)\in S^{m',q'} \textrm{ then there exists }c(x,\xi)\in S^{m+m'-1,q+q'-1} \\
\textrm{such that }[a(x,D),b(x,D)]=c(x,D)\textrm{ moreover}\\
\textrm{there exists }r_0(x,\xi)\in S^{m+m'-2,q+q'-2} \textrm{ such that } c(x,\xi)=\frac1i\{ a,b\}(x,\xi)+r_0(x,\xi) 
\end{array}
\end{equation}
We recall for operators $A$ and $B$ we have $[A,B]=AB-BA$ and $\{a,b\}=\partial_\xi a\partial_xb-\partial_xa\partial_\xi b$. In some cases we have exact formula, for instance $[D,a(x,D)]=\frac1i(\partial_xa)(x,D)$.

In lemma \ref{lemma2} $u$ is real valued but it is convenient to write the integral in the following form
\begin{equation}\label{eq:imaginary-form}
\int (-\DD^\alpha u)u_x \varphi(x) =\im (\varphi(x)Du| \DD^\alpha u)=-\frac{i}2((\DD^\alpha\varphi D-D\varphi\DD^\alpha)u|u).
\end{equation}

Let $\chi\in {\cal C}^\infty_0(\RR)$ such that $0\le \chi \le 1$, $\chi(\xi)=1$ if $|\xi|\le 1$ and $\chi(\xi)=0$ if $|\xi|\ge 2$.
We set 
\begin{equation}\label{eq:T-and-T1-T2}
\begin{array}{l}
 T=\DD^\alpha\varphi D-D\varphi\DD^\alpha=T_1+T_2\textrm{ where}\\
T_1=\DD^\alpha(1-\chi(D))\varphi D-D\varphi (1-\chi(D))\DD^\alpha\\
T_2=\DD^\alpha\chi(D)\varphi D-D\varphi \chi(D)\DD^\alpha
\end{array}
\end{equation}

The proof of lemma \ref{lemma2} follows from \eqref{eq:imaginary-form}, \eqref{eq:T-and-T1-T2} and the two following claims.

\begin{claim}\label{claim-1}
 There exists $C>0$ such that for all $u\in {\cal S}$ we have 
\begin{equation}
 i(T_1u|u)=(\alpha-1)(\phi\DD^\alpha(1-\chi(D))\phi u|u)+R
\end{equation}
where $R$ satisfies $|R|\le C\|\phi u\|^2$.
\end{claim}

\begin{claim}\label{claim-2}
  There exists $C>0$ such that for all $u\in {\cal S}$ we have 
\begin{equation}
 i(T_2u|u)=(\alpha-1)(\phi\DD^\alpha\chi(D)\phi u|u)+R
\end{equation}
where $R$ satisfies $|R|\le C \|\phi u\|^2$.
\end{claim}

\begin{proof}[Proof of Claim  \ref{claim-1}]
In the following we set $a(x,\xi)=\varphi(x)|\xi|^\alpha(1-\chi(\xi))$ and we have $a(x,\xi)\in S^{\alpha,0}$. With this notation we have $T_1=a(x,D)^*D-Da(x,D)$. Following \eqref{Adjoint}, the symbol of $a(x,D)^*$ is $a(x,\xi)+\frac1i \px\pxi a(x,\xi)-\frac12\px^2\pxi^2a(x,\xi)+r_0(x,\xi)$ where $r_0(x,\xi)\in S^{\alpha-3,-3}$. We obtain, following \eqref{Commutator} and remark below, 
\begin{equation}\label{T1-bis}
 \begin{array}{ll}
  T_1&=[a(x,D),D]+\frac1i (\px\pxi a)(x,D)D-\frac12(\px^2\pxi^2)a(x,D)D+r_1(x,D)\\[10pt]
&=i(\px a)(x,D)+\frac1i (\px\pxi a)(x,D)D-\frac12(\px^2\pxi^2a)(x,D)D+r_1(x,D)
 \end{array}
\end{equation}
where $r_1(x,\xi)=r_0(x,\xi)\xi\in S^{\alpha -2,-3}\subset S^{0,-2r}$. We have, by \eqref{ineq-L2}

\begin{equation}\label{Reste-1}
 |(r_1(x,D)u|u)|=|(\langle x\rangle^rr_1(x,D)u|\langle x\rangle^{-r}u)|\le C\left\| \frac{u}{\langle x\rangle^r}\right\|^2
\end{equation}
because $\langle x\rangle^rr_1(x,D)=r_2(x,D)$ where $r_2(x,\xi)\in S^{0,-r}$.

We remark that the symbol of $(\px^2\pxi^2a)(x,D)D$ is real valued, we can apply the following claim.

\begin{claim}\label{claim-3}
 Let $b(x,\xi)\in S^{m,q}$, real valued then there exists $C>0$ such that for all $u\in{\cal S}$, we have
\begin{equation}
 |\im (b(x,D)u|u)|\le C\| \langle x\rangle^{\frac{q-1}2 }\langle D\rangle^{\frac{m-1}2}u\|^2
\end{equation}
\end{claim}

By definition $(T_1u|u)=2i\im(Du|a(x,D)u)$, it is sufficient to consider the imaginary part of the term of \eqref{T1-bis}. In particular $\im ((\px^2\pxi^2a)(x,D)Du|u)$ and we have $(\px^2\pxi^2a)(x,\xi)\xi\in S^{\alpha-1,-2}$. The Claim \ref{claim-3} gives
\begin{equation}\label{Reste-2}
 |\im ((\px^2\pxi^2a)(x,D)Du|u)|\le C\|\langle x\rangle^{-\frac32}\langle D\rangle^\frac{\alpha-2}2u\|^2\le C\left\| \frac{u}{\langle x\rangle^r}\right\|^2
\end{equation}
following \eqref{ineq-L2} and $\langle x\rangle^{-\frac32}\langle \xi\rangle^\frac{\alpha-2}2\in S^{\frac{\alpha-2}2,-\frac32}\subset S^{0,-2r}$.

\begin{proof}[Proof of Claim  \ref{claim-3}]
We have $2i\im(b(x,D)u|u)=((b(x,D)-b(x,D)^*)u|u)$. By \eqref{Adjoint}  we have $b(x,D)^*=b(x,D)+r_0(x,D)$ where $r_0(x,\xi)\in S^{m-1,q-1}$. We have $2i\im(b(x,D)u|u)=(\langle x\rangle^{-\frac{q-1}2}\langle D\rangle^{-\frac{m-1}2}r_0(x,D)u|\langle x\rangle^{\frac{q-1}2}\langle D\rangle^{\frac{m-1}2}u)$ and following \eqref{Composition} $\langle x\rangle^{-\frac{q-1}2}\langle D\rangle^{-\frac{m-1}2}r_0(x,D)=c(x,D)$ where $c(x,\xi)\in S^{\frac{m-1}2,\frac{q-1}2}$. We conclude by \eqref{ineq-L2}.
\end{proof} 

Following \eqref{T1-bis}, \eqref{Reste-1},   \eqref{Reste-2}  and notation of Claim  \ref{claim-1}, we have
\begin{equation}\label{T1-ter}
(T_1u|u)=(i((\px a)(x,D)-(\px \pxi a)(x,D)D)u|u)+R
\end{equation}

We have 
\begin{equation}\label{eq:claim3}
\begin{array}{lll}
(\px a)(x,\xi)-(\px \pxi a)(x,\xi)\xi&=&\varphi'(x)|\xi|^\alpha(1-\chi(\xi))\\
&&-\alpha\varphi'(x)|\xi|^{\alpha-2}|\xi|^2(1-\chi(\xi))+\varphi'(x)|\xi|^\alpha\chi'(\xi)\\
&=&(1-\alpha)\varphi'(x)|\xi|^\alpha(1-\chi(\xi))+\varphi'(x)|\xi|^\alpha\chi'(\xi).
\end{array}
\end{equation}
We have $\varphi'(x)|\xi|^\alpha\chi'(\xi)\in S^{0,-2r}$ because $\chi'$ is compact supported in $\RR\setminus 0$. We have

\begin{equation}\label{Reste-3}
|(\varphi'(x))|D|^\alpha\chi'(D)u|u)|=|(\langle x\rangle^r\varphi'(x))|D|^\alpha\chi'(D)u| \langle x\rangle^{-r} u)|\le C\left\| \frac{u}{\langle x\rangle^r}\right\|^2
\end{equation}
following \eqref{Composition} and \eqref{ineq-L2}. By \eqref{T1-ter},  \eqref{eq:claim3} and \eqref{Reste-3}, we obtain
\begin{equation}\label{T1-quatro}
\begin{array}{ll}
(T_1u|u)&=(1-\alpha)(i\phi^2\DD^\alpha(1-\chi(D))u|u)+R\\
&=(1-\alpha)\big((i\phi\DD^\alpha(1-\chi(D))\phi u|u)+(i\phi[\phi,\DD^\alpha(1-\chi(D))]u|u)\big)+R.
\end{array}
\end{equation}
Following \eqref{Commutator}, we have $i[\phi,\DD^\alpha(1-\chi(D))]=c(x,D)+r_0(x,D)$ where $c(x,\xi)=\{\phi,|\xi|^\alpha(1-\chi(\xi))\}=-\phi'(x)\pxi (|\xi|^\alpha(1-\chi(\xi)))$ and $r_0(x,\xi)\in S^{\alpha-2,-r-2}\subset S^{0,-r}$ then $|(r_0(x,D)u|\phi  u)|\le C\| \langle x\rangle^{-r}u\|^2$.
We have $\phi(x) c(x,\xi)\in S^{\alpha-1,-2r-1}\subset S^{1,-2r+1}$ and real valued, we can apply the Claim \ref{claim-3} to obtain $|\im (\phi(x) c(x,D)u|u)|\le C\| \langle x\rangle^{-r}u\|^2$. With \eqref{T1-quatro}, this proves Claim \ref{claim-1}.
\end{proof}

\begin{proof}[Proof of Claim  \ref{claim-2}]
Since $[D,a(x,D)]= \frac 1i (\partial_x a)(x,D)$ for any $a(x,D)$,
\begin{equation}\label{eq1:claim2}
 \begin{array}{ll}
 T_2&=\DD^\alpha D\chi(D)\varphi(x)-\varphi(x)\DD^\alpha D\chi(D)+i\DD^\alpha\chi(D)\varphi'(x)+i\varphi'(x)\DD^\alpha\chi(D)\\
&=[\DD^\alpha D\chi(D),\varphi(x)]+2i\phi \DD^\alpha \chi(D)\phi +i\big[ [ \DD^\alpha \chi(D),\phi],\phi\big]=A_1+A_2+A_3.
\end{array}
\end{equation}

We remark that $D\DD^\alpha\chi(D)u=g*u$ where $\hat g(\xi)=|\xi|^\alpha\xi\chi(\xi)$.

\begin{claim}\label{claim-4}
 Let $A_1=[\DD^\alpha D\chi(D),\varphi]$, then there exists $C>0$ such that for all $u\in{\cal S}$,
\begin{equation}\label{eq:claim4}
 i(A_1u|u)=(\alpha+1)(\phi\DD^\alpha\chi(D)\phi u|u)+R
\end{equation}
where $|R|\le C\| \langle x\rangle^{-r}u\|^2$.
In particular,
\begin{equation}\label{eq:claim4b}
 i((A_1+A_2)u|u)=(\alpha-1)(\phi\DD^\alpha\chi(D)\phi u|u)+R.
\end{equation}
\end{claim}
\begin{proof}[Proof of Claim  \ref{claim-4}]
We have, by a direct computation $[\DD^\alpha D\chi(D),\varphi]u(x)=\int g(x-y)(\varphi(y)-\varphi(x))u(y)dy$.
To prove Claim   \ref{claim-4} we need the following two claims, proved below.
\begin{claim}\label{claim-5}
 There exists $C>0$ such that, we have
\begin{equation}
 \varphi(y)-\varphi(x)=\frac{y-x}{\langle x\rangle^r\langle y\rangle^r}+Q(x,y)
\end{equation}
where $Q(x,y)$ satistifies
\begin{equation}\label{eq1:claim-5}
 |Q(x,y)|\le C\frac{|x-y|^2}{(\langle x\rangle+\langle y\rangle)^{2r+1}} \textrm{ if }|x-y|\le \frac12(\langle x\rangle+\langle y\rangle)
\end{equation}
\begin{equation}\label{eq2:claim-5}
  |Q(x,y)|\le C+C\frac{|x-y|}{\langle x\rangle^r\langle y\rangle^r}\textrm{ if }|x-y|\ge \frac12(\langle x\rangle+\langle y\rangle)
\end{equation}
\end{claim}
We remark that if $|x-y|\le \frac12(\langle x\rangle+\langle y\rangle)$ then $\langle x\rangle\sim\langle y\rangle$ and if $|x-y|\ge \frac12(\langle x\rangle+\langle y\rangle)$ then $\langle x-y\rangle\sim |x-y|\sim \langle x\rangle+\langle y\rangle$.

\begin{claim}\label{claim-6}
 Let $Ku(x)=\int Q(x,y)g(x-y)u(y)dy$, there exists $C>0$ such that for all $u\in {\cal S}$ we have,
\begin{equation}
 |(Ku|u)|\le C \| \langle x\rangle^{-r}u\|^2
\end{equation}
\end{claim}

Following Claims \ref{claim-5} and \ref{claim-6}, we have $A_1u=\phi(h*(\phi u))+Ru$ where $|(Ru|u)|\le C\| \langle x\rangle^{-r}u\|^2$ and $h(x)=-xg(x)$. By definition of $g$ we have
\begin{equation}
 \begin{array}{lll}
 h(x)&=&{\displaystyle \frac1{2\pi}\int -xe^{ix\xi}\xi|\xi|^\alpha\chi(\xi)d\xi}\\[12pt]
&=&{\displaystyle \frac{i}{2\pi}\int\pxi(e^{ix\xi})\xi|\xi|^\alpha\chi(\xi)d\xi}\\[12pt]
&=&{\displaystyle \frac{-i}{2\pi}\int e^{ix\xi}\pxi \big( \xi|\xi|^\alpha\chi(\xi) \big) d\xi}.
 \end{array}
\end{equation}
In the last equality we use  that $\xi|\xi|^\alpha$ is a $C^1$ function, and we have $\pxi \big( \xi|\xi|^\alpha\chi(\xi) \big)= (\alpha+1)|\xi|^\alpha\chi(\xi)+\xi|\xi|^\alpha\chi'(\xi) $. Then we have $h(x)=h_1(x)+h_2(x)$ where $\hat h_1(\xi)=-i(\alpha+1)|\xi|^\alpha\chi(\xi)$ and $\hat h_2(\xi)=-i\xi|\xi|^\alpha\chi'(\xi)$. We have $\phi(h_1*(\phi u))=-i(\alpha+1)(\phi\DD^\alpha\chi(D)\phi u)(x)$. This term gives the first term of the right hand side of \eqref{eq:claim4}. We have $\phi(h_2*(\phi u))=(\phi D\DD^\alpha\chi'(D)\phi u)(x)$ and by \eqref{Composition}, $D\DD^\alpha\chi'(D)\phi$ is an operator with symbol in $S^{0,-r}$ (we recall $\chi'$ is supported in $1\le |\xi|\le 2$), we have by \eqref{ineq-L2}, $|(D\DD^\alpha\chi'(D)\phi u|\phi u)|\le C \| \langle x\rangle^{-r}u\|^2$. This proves Claim \ref{claim-4}.
\end{proof}

\begin{proof}[Proof of Claim  \ref{claim-5}]
By definition $\varphi $ is bounded then \eqref{eq2:claim-5} is obvious. We have $\varphi(y)-\varphi(x)=\int_x^y\frac1{\langle s\rangle^{2r}}ds$ then $Q(x,y)=\int_x^y\left( \frac1{\langle s\rangle^{2r}}-\frac1{\langle x\rangle^{r}\langle y\rangle^{r}}\right) ds$. We have 
\begin{equation}
\frac1{\langle s\rangle^{2r}}-\frac1{\langle x\rangle^{r}\langle y\rangle^{r}}=\frac1{\langle s\rangle^{r}}\left(  \frac1{\langle s\rangle^{r}}-\frac1{\langle x\rangle^{r}} \right) +\frac1{\langle x\rangle^{r}}\left(  \frac1{\langle s\rangle^{r}}-\frac1{\langle y\rangle^{r}} \right)
\end{equation}
We have $\langle s\rangle\le \langle x\rangle+\langle y\rangle$ because $s\in[x,y]$, and $\langle s\rangle\ge \inf(\langle x\rangle,\langle y\rangle)\sim \langle x\rangle\sim\langle y\rangle\sim \langle x\rangle+\langle y\rangle$ if $|x-y|\le \frac12(\langle x\rangle+\langle y\rangle)$. To prove  \eqref{eq1:claim-5}, it is sufficient to prove, 
\begin{equation}\label{eq3:claim-5}
\left|\frac1{ \langle s\rangle^{r}}-\frac1{\langle x\rangle^{r}}\right| \le C\frac{|s-x|}{\langle x\rangle^{r+1}}
. \end{equation}
Writing $\frac1{ \langle s\rangle^{r}}-\frac1{\langle x\rangle^{r}}=\int_x^s\psi(t)dt$ where $\psi(t)=\partial_t(\frac1{ \langle t\rangle^{r}})$, we have $|\psi(t)|\le C\frac1{ \langle t\rangle^{r+1}}$, this gives \eqref{eq3:claim-5}.
\end{proof}

\begin{proof}[Proof of Claim  \ref{claim-6}]
Writing $((Ku)(x)|u(x))=(\langle x\rangle^rK(\langle y\rangle^r\langle y\rangle ^{-r}u)(x) |\langle x\rangle^{-r}u(x))$, it is sufficient to prove that $\langle x\rangle^rK(\langle y\rangle^rv)(x)$ defines a bounded operator on $L^2$. The kernel of this operator is $H(x,y)=\langle x\rangle^r\langle y\rangle^rQ(x,y)g(x-y)=H_1(x,y)+H_2(x,y)$, where $H_1$ and $H_2$ are $H$ restricted respectively to the regions $|x-y|\le \frac12(\langle x\rangle +\langle y\rangle)$ and $|x-y|\ge \frac12(\langle x\rangle +\langle y\rangle)$. Following Lemma \ref{LE:decay-2}, we have $|g(x-y)|\le \frac{C}{\langle x-y\rangle^{\alpha+2}}$.

From Claim \ref{claim-5} we have
\begin{equation}\label{eq1:claim6}
 \begin{array}{lll}
 |H_1(x,y)|&\le& C\dfrac{\langle x\rangle^r\langle y\rangle^r|x-y|^2}{\langle x-y\rangle^{\alpha+2}(\langle x\rangle+\langle y\rangle)^{2r+1}}\\[12pt]
&\le& \dfrac{C}{\langle x-y\rangle^{\alpha}(\langle x\rangle+\langle y\rangle)}\\[12pt]
&\le& \dfrac{C}{\langle x-y\rangle^{\alpha+1}}
\end{array}
\end{equation}
and 
\begin{equation}\label{eq2:claim6}
  \begin{array}{lll}
   |H_2(x,y)|&\le& C\dfrac{\langle x\rangle^r\langle y\rangle^r}{\langle x-y\rangle^{\alpha+2}}\left(C+\dfrac{C|x-y|}{\langle x\rangle^r\langle y\rangle^r} \right) \\[12pt]
&\le &  C\dfrac{\langle x\rangle^r\langle y\rangle^r}{\langle x-y\rangle^{\alpha+2}}+\dfrac{C}{\langle x-y\rangle^{\alpha+1}}=H_3(x,y)+H_4(x,y)
  \end{array}
\end{equation}

We claim $\int H_3(x,y) dy \leq C$ (and  by symmetry $\int H_3(x,y) dx \leq C$). Indeed,
\begin{align*}
\int H_3(x,y) dy & \leq \int_{|y|<|x|} H_3(x,y) dy + \int_{|y|>|x|} H_3(x,y) dy \\
& \leq C \langle x\rangle^{r-(\alpha +2)} \int_{|y|<|x|} \langle y\rangle^r dy
+ C \langle x\rangle^r \int_{|y|>|x|} \langle y\rangle^{r-(\alpha +2)} dy
\leq C  \langle x\rangle^{2 r - (\alpha +1)} \leq C.
\end{align*}
The same estimate is trivially true for $H_1$ and $H_4$.
Thus, by   Schur's lemma, the operator with kernel $H$ is bounded on $L^2$.
\end{proof}

\begin{claim}\label{claim-7}
 Let $A_3=i \big[[\DD^\alpha\chi(D),\phi],\phi\big]$, there exists $C>0$ such that for all $u\in {\cal S}$ we have,
\begin{equation}
  |(A_3u|u)|\le C \| \langle x\rangle^{-r}u\|^2
\end{equation}
\end{claim}

\begin{proof}[Proof of Claim  \ref{claim-7}]
We set $h(x)=\frac1{2\pi}\int e^{ix\xi}|\xi|^\alpha\chi(\xi)d\xi$. Following Lemma \ref{LE:decay-2}, there exists $C>0$ such that $| h(x)|\le\frac{C}{\langle x\rangle^{\alpha+1}}$. We have $\big[[\DD^\alpha\chi(D),\phi],\phi\big]u=\int h(x-y)\big(\phi(x)-\phi(y)\big)^2u(y)dy$. We need the following Claim to continue.

\begin{claim}\label{claim-8}
 There exists $C>0$ such that
\begin{align*}
 |\phi(x)-\phi(y)|\le C\frac{|x-y|}{(\langle x\rangle+\langle y\rangle)^{r+1}}  \textrm{ if }|x-y|\le \frac12(\langle x\rangle+\langle y\rangle)\\
|\phi(x)-\phi(y)|\le \frac1{\langle x\rangle^r}+\frac1{\langle y\rangle^r}\textrm{ if }|x-y|\ge \frac12(\langle x\rangle+\langle y\rangle)
\end{align*}
\end{claim}
\begin{proof}[Proof of Claim  \ref{claim-8}]
We have $\phi(x)-\phi(y)=\int_y^x\zeta(s)ds $ where $\zeta=\partial_s\left( \frac{1}{\langle s \rangle^r}\right) $  and we have $|\zeta(s)|\le \frac{C}{\langle s \rangle^{r+1}}$. If $|x-y|\le \frac12(\langle x\rangle+\langle y\rangle)$, we have $\langle s \rangle\sim\langle x \rangle\sim \langle y \rangle\sim \langle x \rangle +\langle y \rangle$, this gives the first inequality. The second one is obvious.
\end{proof}

 We argue as in the proof of Claim \ref{claim-6}.  Let $R(x,y)=\langle x\rangle^rh(x-y)\big(\phi(x)-\phi(y)\big)^2\langle y\rangle^r=R_1(x,y)+R_2(x,y)$ where $R_1$ and $R_2$ are $R$ restricted respectively to the regions $|x-y|\le \frac12(\langle x\rangle+\langle y\rangle)$ and $|x-y|\ge \frac12(\langle x\rangle+\langle y\rangle)$. It is sufficient to prove that $R$ defines an bounded operator on $L^2$.

 We have
\begin{equation}
 \begin{array}{lll}
 |R_1(x,y)|&\le& C\dfrac{\langle x\rangle^r\langle y\rangle^r|x-y|^2}{\langle x-y\rangle^{\alpha+1}(\langle x\rangle+\langle y\rangle)^{2r+2}}\\[12pt]
&\le&\dfrac{C}{\langle x-y\rangle^{\alpha+1}}
\end{array}
\end{equation}
And  
\begin{equation}
 \begin{array}{lll}
  |R_2(x,y)|&\le& C\dfrac{\langle x\rangle^r\langle y\rangle^r}{\langle x-y\rangle^{\alpha+1}} \left(    \dfrac{1}{\langle x\rangle^{2r}}+ \dfrac{1}{\langle y\rangle^{2r}} \right) \\[12pt]
&\le&   \dfrac{C\langle x\rangle^r}{\langle x-y\rangle^{\alpha+1}\langle y\rangle^r}+\dfrac{C\langle y\rangle^r}{\langle x-y\rangle^{\alpha+1}\langle x\rangle^r} =R_3(x,y)+R_4(x,y)\\[12pt]
 \end{array}
\end{equation}
By symmetry, it is now sufficient to prove that $R_3$ defines a bounded operator on $L^2$. We have 
\begin{equation}
 \int R_3(x,y)\langle x\rangle^{-\frac12}dx\le C\langle y\rangle ^{-r}\textrm{ and } \int R_3(x,y)\langle y\rangle^{-r}dy\le C\langle x\rangle ^{-\frac12}
\end{equation}
if $r\le \frac{\alpha +1}{2}$. Using a variant of Schur's lemma (see e.g. Theorem 5.2 in \cite{HS}), the operator with kernel $R(x,y)$ is bounded on $L^2$.
\end{proof}

By \eqref{eq1:claim2}, the claims \ref{claim-4} and \ref{claim-7} we obtain $i(T_2u|u)=(\alpha-1)(\phi\DD^\alpha\chi(D)\phi u|u)+R$ where $R $ sastifies the required estimates to prove Claim \ref{claim-2}.
\end{proof}

  Lemma \ref{lemma2} follows from the Claims \ref{claim-1} and \ref{claim-2}.
\end{proof}
\begin{proof}[Proof of Lemma \ref{lemma1}]
 We have 
\begin{equation}
 \int (-\DD^\alpha u)u\varphi'dx=(-\phi^2\DD^\alpha u|u)=(-\phi\DD^\alpha\phi u|u)-(\phi[\phi,\DD^\alpha]u|u)
\end{equation}
As the left hand side is real, we can take the real part of the last term and we have,
\begin{equation}
\begin{array}{lll}
 2\re (\phi[\phi,\DD^\alpha]u|u)&=&(\phi[\phi,\DD^\alpha]u|u)+(u|\phi [\phi,\DD^\alpha]u)\\
&=&(\phi[\phi,\DD^\alpha]u|u)-([\phi,\DD^\alpha]\phi u|u)=(\big[\phi,[\phi,\DD^\alpha]\big]u|u)
\end{array}
\end{equation}
By pseudodifferential calculus \eqref{Commutator}, the symbol of $\big[\phi,[\phi,\DD^\alpha(1-\chi(D))]\big]$ is in $S^{\alpha-2,-2r-2}\subset S^{0,-2r}$ and then it satisfies 
\begin{equation}
| (\big[\phi,[\phi,\DD^\alpha(1-\chi(D))]\big]u|u)|\le C\| \langle x\rangle^{-r}u\|^2
\end{equation}

The term $(\big[\phi,[\phi,\DD^\alpha\chi(D)]\big]u|u)\le C\| \langle x\rangle^{-r}u\|^2$ by Claim \ref{claim-7}. This proves that 
\begin{equation}
  \int (-\DD^\alpha u)u\varphi'dx\le -\|\DD^\frac{\alpha}{2}(\phi u)\|^2 +C\| \langle x\rangle^{-r}u\|^2,
\end{equation}
and completes the proof of Lemma \ref{lemma1}.
\end{proof}

\subsection{Monotonicity result on $\eta(t)$}

For future use, we also state a monotonicity result for $\eta(t)$, restricted to the regular regime, i.e. the situation where the solution stays close to a fixed soliton.

\begin{proposition}\label{PR:monoeta}
Let $r\in (\frac 12 ,  \frac 12 (\alpha+1)]$ and $0<\mu<1$.
	Under the assumptions of Lemma \ref{MODULATION}, 
	with the restriction $\lambda_0(t)=1$,
for 	$\epsilon_0=\epsilon_0(\mu,r)$ small enough and $A=A(\mu,r)$ large enough, there exists $C=C(\mu,r,A)>0$ such that for all $x_0>1$, 
  	\begin{equation}\label{monotonicity3}\begin{split}
	&\int \eta^2(s_2,y)\left[\varphi_A(\lambda^{\frac 2 \alpha}(s_2) y-x_0)
	-\varphi_A(-x_0)\right] dy
	 \\&  \leq \int \eta^2(s_1,y)\left[\varphi_A(\lambda^{\frac 2 \alpha}(s_1) y-x_0 -\mu (s_2-s_1))  
	  -  \varphi_A(-x_0 -\mu (s_2-s_1)  )\right] dx\\ & +C \int_{s_1}^{s_2} \frac {\|\eta(s)\|_{L^2}^2}  {(x_0 + \mu(s_2-s))^{2r}} ds.
	\end{split}\end{equation}  
 \end{proposition}
\begin{proof}[Sketch of proof]
Using Lemmas \ref{lemma2}--\ref{lemma1}, the proof is similar to the one of Proposition 2 in \cite{KM}, the only difference being the additional scaling parameter $\lambda(s)$ (close to $1$) in the present situation.
Let 
$$
	\tilde y = \lambda^{\frac 2 \alpha}(s)y-x_0 - \mu (s_2-s),
	\quad M_{\eta}(s) = \frac 12 \int \eta^2(s) \left[\varphi_A( \tilde y)
	- \varphi_A( -x_0-\mu(s_2-s))\right].
$$
Using the equation of $\eta(s)$ (see Lemma \ref{MODULATION}), Lemmas \ref{lemma2}--\ref{lemma1} and estimates on 
$\varphi_A$, as in \cite{KM}, one finds
$$
M'_{\eta}(s)\leq  \frac {C \|\eta(s)\|_{L^2}^2} {(x_0 + \mu(s_2-s))^{2r}} ,
$$
and the result follows by integration on $[s_1,s_2]$.
\end{proof}
\section{Nonlinear Liouville property and asymptotic stability}
This section is devoted to the regular regime: we study rigidity properties of the nonlinear equation \eqref{dgBO} in a neighborhood of a soliton. In this section, $\alpha_0<\alpha<2$, where $\alpha_0$ is given by Proposition \ref{pr:close2} and $Q$ denotes the only ground state solution of \eqref{eq:Q}.
Note that we could also work with a general $1\leq \alpha <2$, assuming the linear Liouville property.

\subsection{Nonlinear Liouville property}
\begin{proposition}[Nonlinear Liouville property]\label{PR:4}
 Let $\alpha_0<\alpha<2$.
There exists $\epsilon>0$ such if $u(t)$ is a global ($t\in \RR$) solution of \eqref{dgBO} satisfying
for some $x_0(t)$,
\begin{align}
& \forall t\in \RR, \quad  \|u(t)-Q(.-x_0(t))\|_{H^{\frac{\alpha}{2}}}\leq \epsilon,\label{eq:a1}\\
& \forall \delta>0, \ \exists B>0,\ \forall t\in \RR,   \ \quad
\int_{|x|> B} |u(t,x-x_0(t))|^2 dx \leq \delta,\label{eq:a2}
\end{align}
then $u(t,x)\equiv Q_{\lambda_0}(x-x_0-\lambda_0^{-2}t)$ for some $x_0 \in \RR$ and some $\lambda_0$ close to $1$.
\end{proposition}

\begin{proof}
The proof of Proposition \ref{PR:4} is by contradiction. Assume that there exists a sequence $u_n(t)$ of global
$H^{\frac \alpha 2}$ solutions of \eqref{dgBO} close to a translation of $Q$ for all time and such that their decomposition parameters $\eta_n(t)$, $\lambda_n(t)$, $\rho_n(t)$ given by Lemma \ref{MODULATION} satisfy
\begin{align}
& \sup_{s\in \RR} \left(|\lambda_n(s)-1|+ \|\eta_n(s)\|_{H^{\frac \alpha  2}} \right)
\to 0\quad \text{as $n\to +\infty$},\label{eq:st1}\\
& \eta_n \not \equiv 0, \label{eq:st2}\\
& \forall n, \ \forall \delta,\ \exists B_{n,\delta}>0, \ \forall t\in \RR,\ \quad
\int_{|x|> B_{n,\delta}} |u_n(t,x+\rho_n(t))|^2 dx \leq \delta.\label{eq:st3}
\end{align}
We follow the strategy of \cite{KM}, proof of Theorem 2. Define
$
	0\not \equiv b_n = \sup_{s\in \RR} \|\eta_n(s)\|_{L^2}$, $b_n \to 0$ as $n\to +\infty$.
Then, there exists $s_n$ such that $\|\eta_n(s_n)\|_{L^2} \geq \frac 12 b_n$. We set
$$
	w_n(s,y) = \frac {\eta_n(s_n+s,y)}{b_n},
$$
and we claim the following convergence result for the sequence $(w_n)$.
\begin{lemma}\label{PR:5}
	There exists a subsequence of $(w_n)$, denoted $(w_{n'})$ and $w\in C(\RR,L^2(\RR))\cap L^{\infty}(\RR,L^2(\RR))$ such that
	$$
		\forall s\in \RR,\quad w_{n'}(s) \rightharpoonup w(s) \quad \text{in $L^2(\RR)$ weak as $n\to +\infty$.}
	$$
	Moreover, $w(s)$  satisfies for some continuous functions $\beta(s)$, $\gamma(s)$,
	\begin{align*}
		& w_s = (L w)_y + \beta(s) Q' + \gamma(s) \Lambda Q \quad  \text{on $\RR\times \RR$},\\
		& w\neq 0,\quad \int \chi_0 w = \int Q'w =0,\\
		& \forall s\in \RR,\forall y_0>1,\quad \int_{|y|>y_0} w^2(s,y) dy \leq \frac C{y_0^{{\alpha}}},
	\end{align*}
\end{lemma}
\begin{proof}[Sketch of the proof of Lemma \ref{PR:5}]
We proceed as in \cite{KM}, proof of Proposition 5.

\noindent\emph{Decay estimate.} From Proposition \ref{PR:monoeta} (with $r=\frac{\alpha+1}{2}$,
$s_2=s$ and $s_1\to -\infty$) and \eqref{eq:st3}, it follows that   
\begin{equation}\label{eq:dec10}
	\forall y_0>1,\forall s \in \RR,\quad \int_{|y|>y_0} \eta_n^2(s,y)dy\leq \frac {C b_n^2}{y_0^{\alpha}},\quad
	\int_{|y|>y_0} w_n^2(s,y) dy \leq \frac C{y_0^{\alpha}}.
\end{equation}

\noindent\emph{Local smoothing estimate.} As in \cite{KM}, we obtain using the equation of $w_n(s)$
\begin{equation}\label{eq:lse}
	\int_0^1 \int |D^{\frac \alpha 2}( w_n(s,y) \sqrt{\varphi'(y)}|^2 dy ds \leq C.
\end{equation}

\noindent\emph{Compactness in $L^2$.} Following \eqref{eq:dec10} and \eqref{eq:lse}, there exists $\tau_n\in [0,1]$ and a subsequence of $(w_n)$  
still denoted by $(w_n)$, $s_0\in [0,1]$  and $w_{s_0}\in L^2$ such that
$$
	w_n(\tau_n) \to w_{s_0} \quad \text{in $L^2$},
	\quad \tau_n \to s_0 \quad \text{as $n\to +\infty$.}
$$
Moreover,
$
	\int w_{s_0} Q' = \int w_{s_0} \chi_0 = 0.
$

Next, note that
\begin{align*}
	w_{ns} & = \partial_y(L w_n) - \partial_y \left(\frac 1{b_n} \mathcal{R}(b_n w_n)\right)
	+\frac 1{b_n} \frac {\lambda_{ns}}{\lambda_n} (\Lambda Q + b_n \Lambda w_n)
	+ \frac 1{b_n} \left(\frac {\rho_{ns}}{\lambda_n^{\frac 2 \alpha}} -1\right) \partial_y(Q + b_n w_n)\\
	&= \partial_y(L w_n) - \partial_y \left(\frac 1{b_n} \mathcal{R}(b_n w_n)\right)
	+\beta_n Q' + \gamma_n \Lambda Q + b_n F_n' + b_n G_n + b_n \tilde \beta_n w_{ny} 
	+ b_n \tilde \gamma_n \Lambda w_n,
\end{align*}
where
\begin{align*}
	&	\beta_n 	= \frac 1 {\int (Q')^2} \int w_n L(Q''),\quad
	\tilde \beta_n  = \frac 1 {b_n} \left(\frac {\rho_{ns}}{\lambda_n^{\frac 2 \alpha}} -1\right),
	\quad F_n = \frac 1{b_n} (\tilde \beta_n - \beta_n) Q,\\
	& 	\gamma_n	= \frac 1{\int \Lambda Q \chi_0} \int w_n L(\chi_0'),\quad
	\tilde \gamma_n = \frac 1{b_n} \frac {\lambda_{ns}}{\lambda_n},\quad
	G_n = \frac 1{b_n} (\tilde \gamma_n - \gamma_n) \Lambda Q.
\end{align*}
Set
$$
\tilde w_n(s) = w_n(s) - \Lambda Q \int_{\tau_n}^s \gamma_n(s') ds'
- Q' \int_{\tau_n}^s \left(\beta_n(s') + 2 \int_{\tau_n}^{s'} \gamma_n(s'') ds''\right) ds',
$$
then
$$
\tilde w_{ns}  = \partial_y(L \tilde w_n) - \partial_y \left(\frac 1{b_n} \mathcal{R}(b_n w_n)\right)
	 + b_n F_n' + b_n G_n + b_n \tilde \beta_n w_{ny} 
	+ b_n \tilde \gamma_n \Lambda w_n,
$$
Consider $\tilde w(s,y)$ the unique global solution of 
$$
	\tilde w_{s}  = \partial_y(L \tilde w) \quad \text{on $\RR\times \RR$},\quad
	\tilde w(s_0)=w_{s_0} \quad \text{on $\RR$}.
$$
Then (see proof of Lemma 9 in \cite{KM}), we have
$$
\forall s\in \RR,\quad 	\tilde w_n(s) \rightharpoonup \tilde w(s) \quad \text{in $L^2$ weak.}
$$
Finally, Lemma \ref{PR:5} is proved with
$$
	w(s,y)= \tilde w(s,y) + \Lambda Q \int_{s_0}^s   \gamma(s') ds'
+ Q' \int_{s_0}^s \left(  \beta(s') + 2 \int_{s_0}^{s'}   \gamma(s'') ds''\right) ds'
$$
where
$$
 \gamma(s)= \frac 1{\int \Lambda Q \chi_0} \int \tilde w L(\chi_0'),\quad
 \beta(s)= \frac 1 {\int (Q')^2} \int \left(\tilde w + \Lambda Q \int_{s_0}^s \gamma(s') ds'\right) L(Q'').
$$
\end{proof}

We finish the proof of Proposition \ref{PR:4} by observing that  the function $w(s,y)$   constructed
in Lemma \ref{PR:5} contradicts the linear Liouville property, thus reaching the desired contradiction.
Indeed, using the strategy of the proof of Corollary 1 in \cite{yvanSIAM}, we obtain
$$
	w(s,y) = a(t) \Lambda Q + b(t) Q'.
$$
But since $\int w\chi_0=\int w Q' = 0$, we obtain $a(t)=b(t)\equiv 0$ and thus $w\equiv 0$, which
is a contradiction.
\end{proof}

\subsection{Asymptotic stability in the bounded regime}

The next proposition is not used in the proof of Theorem \ref{TH2} but it is stated as a consequence of Proposition \ref{PR:4} and the monotonicity arguments of Section 4.

 \begin{proposition}[Asymptotic stability]\label{PR:6}
Assume $\alpha_0<\alpha\le 2$.
There exists $\epsilon>0$ such if $u(t)$ is a global ($t\in \RR$) solution of \eqref{dgBO} satisfying
\begin{equation}\label{eq:89}
 \forall t\in \RR, \quad  \inf_{x_0\in \RR}\|u(t)-Q(.-x_0)\|_{H^\frac{\alpha}{2}}\leq \epsilon,
\end{equation}
then there exist $\lambda(t)>0$, $\rho(t)\in \RR$ such that
$$
	\eta(t,y)= \lambda^{\frac 1\alpha}(t) u\left(t,\lambda^{\frac 2 \alpha}(t) y+\rho(t)\right)-Q(y)
$$
satisfies
$$
	\eta(t)\rightharpoonup 0 \quad \text{in $H^{\frac \alpha 2}$ as $t\to +\infty$.}
$$

\end{proposition}

Except for the presence of the scaling parameter, it is similar to the proof of Theorem 2 from Theorem 1 in \cite{KM}. It is also close to the original proof for the gKdV equation in \cite{MMjmpa}. We thus omit the proof.

\section{Finite or infinite time blow up in the energy space}

In this section, we prove Theorem \ref{TH2} following the strategy of \cite{Me} and using the classification result given by Proposition \ref{PR:4}.

\medskip

Let  $\alpha \in (\alpha_0,2]$ where $\alpha_0$ be given by Proposition \ref{pr:close2}. Consider  
 an initial data $u(0) \in H^{\frac \alpha 2}(\RR)$ such that
$$
	E(u(0))<0 \quad \text{and} \quad  0<\beta(u(0))=\int u^2(0) -\int Q^2 <¬†\beta_0,
$$
where $\beta_0$ is small enough (to be chosen) and  $u(t)$  the corresponding solution of \eqref{dgBO}.
 Let $[0,T)$, $0<T\leq +\infty$ be the maximal interval of existence of $u(t)$  as a solution of \eqref{dgBO} in $H^{\frac \alpha 2}$ (for $t\geq 0$).

We need the following variational result concerning negative energy $H^{\frac \alpha 2}$ functions, with $L^2$ norm close to the $L^2$ norm of $Q$. 

\begin{lemma}\label{claim10}
There exists $\beta_0>0$ such that for all $v\in H^{\frac \alpha 2}$, if $E(v)<0$ and
$\beta(v)< \beta_0$ then there exists $x_0 \in \RR$, $\lambda_0>0$, $\epsilon=\pm 1$ such that
$$
\|Q - \epsilon \lambda_0^{\frac 1 \alpha} v (\lambda_0^{\frac 2 \alpha} (x+x_0)) \|_{H^{\frac \alpha 2}}
\leq \delta(\beta),
$$
where $\delta(\beta)\to 0$ as $\beta\to 0$.
\end{lemma}

We omit the proof since it is similar to the one of Lemma 1 in \cite{Me}, using \eqref{varchar}.

\medskip

By conservation of mass, of energy and under the assumptions on $u(0)$, for $\beta_0$ small enough, it follows from Lemma \ref{claim10} applied to $u(t)$ for all $t\in [0,T)$, that 
 $u(t)$ is close to $\pm Q_{\lambda_0(t)}(x-\rho_0(t))$ for some $\lambda_0(t)$, $\rho_0(t)$. 
Without loss of generality, and by continuity in $H^{\frac \alpha 2}$, we assume that $u$ is close to $+Q$ (up to scaling and translation), by possibly considering $-u$ instead of $u$ and using the invariance of the equation.

Now, from Lemma \ref{MODULATION}, possibly taking $\beta_0$ smaller, there exist $\lambda(t)$, $\rho(t)$ on $[0,T)$ such that, for all $t\in [0,T)$,
$$
	\eta(t,y) = \lambda^{\frac 1\alpha}(t) u(t,\lambda^{\frac 2 \alpha}(t) y + \rho(t)) - Q(y)
$$
satisfies
\begin{align}
& \int Q'(y) \eta(t,y)dy = \int \chi_0(y) \eta(t,y) dy =0	\\
&	\|\eta(t)\|_{H^{\frac \alpha 2}} \leq C \sqrt{\beta(u(0))},
\label{eq1:le11}
\\
& \left|\frac{\lambda_s}{\lambda} \right|
+ \left| \left(\frac {\rho_s}{\lambda^{\frac 2 \alpha}} -1\right) \right|  \leq C \sqrt{ \beta(u(0))}.
\label{eq2:le11}
\end{align}
Note that Lemmas \ref{claim10} and \ref{MODULATION} only give $\|\eta\|_{H^{\frac \alpha 2}} \leq C \delta(\beta(0))$, where $\delta(\beta)$ is defined in Lemma \ref{claim10}, but not explicit.
Actually, in this context, this estimate can be   refined to get \eqref{eq1:le11}
by using energy arguments, exactly as in the proof of Lemma 3 in \cite{Me}.
\medskip

Now, we prove that
\begin{itemize}
	\item Either the solution $u(t)$ ceases to exist in finite time $0<T<+\infty$ and consequently   by Theorem~\ref{TH1}, $\lim_{t\to T} \| u(t) \|_{H^{\frac \alpha 2}} = +\infty$.
	\item Or it exists for all time and then $\lim_{t\to +\infty} \|u(t)\|_{H^{\frac \alpha 2}} = +\infty$.
\end{itemize}
The proof is by contradiction. Assume on the contrary that the solution
$u(t)$ is globally defined in $H^{\frac \alpha 2}$ for $t\geq 0$ and that there exists an increasing sequence $\bar t_m \to +\infty$ and $c_0>0$ such that
\begin{equation}\label{eq1:pth2}
	\|u(\bar t_m)\|_{H^{\frac \alpha 2}} \leq c_0.
\end{equation}
We proceed in four steps to reach a contradiction.

\medskip

\noindent\textbf{Step 1.} Renormalisation and reduction of the problem.
We recall that $\| u(t)\|_{L^2}$ is bounded and we define   
$$
	\ell = \liminf_{t\to +\infty} \| |D|^{\frac \alpha 2} u_n (t) \|_{L^2} < \infty. 
$$
Note first that $\ell > 0$. Indeed, for all time $t$, 
$\int |u(t)|^{2\alpha +2} > -(\alpha +1)(2 \alpha +1) E(u(0)) >0$ and   by the Gagliardo-Nirenberg inequality \eqref{gn}, we obtain
$\ell >0$.
From the definition of $\ell$, there exists $t_{0}$ such that
$$
\||D|^{\frac \alpha 2} u (t_{0}) \|_{L^2} \leq \ell (1+ \beta_{0})
\quad \text{and} \quad
\forall t\geq t_{0}, \ \||D|^{\frac \alpha 2} u (t) \|_{L^2}
\geq \ell (1-\beta_{0}).
$$
We consider the following rescalled version of $u(t,x)$: let $\bar \lambda = \frac {\|\DD^{\frac \alpha 2} Q\|_{L^2}}
	{\ell}$ and
$$
	\bar u(t,x) =  \bar \lambda ^{\frac 1 \alpha}
	u\left(  \bar \lambda ^{2+\frac 2 \alpha} t + t_{0},
	 \bar \lambda ^{\frac 2 \alpha}x\right).
$$
Note that  $\|Q\|_{L^2}^2 < \|\bar u(0)\|_{L^2}^2 < \|Q\|_{L^2}^2 + \beta_0$,
$E(\bar u(0))<0$, $\beta(\bar u(0)) < \beta_{0}$, $\bar u(t)$ is still 
 a solution of \eqref{dgBO} in $H^{\frac \alpha 2}$ defined for all $t\geq 0$, and
for all $t\geq 0$, $\|\DD^{\frac \alpha 2}\bar u(t)\|_{L^2}\geq (1-\beta_{0}) \|\DD^{\frac \alpha 2} Q\|_{L^2}$.
Moreover, there exists a sequence $t_{m}\to +\infty$, such that
$$
\lim_{{m\to +\infty}} \|\DD^{\frac \alpha 2} \bar u(t_{m}) \|_{L^2}
= \|\DD^{\frac \alpha 2} Q \|_{L^2}\quad \text{and}\quad 
\lim_{m\to +\infty}t_{m+1}-t_{m} = +\infty.
$$
Let $\bar \eta(t)$, $\bar \lambda(t)$ and $\bar \rho(t)$ be the parameters of the decomposition of $\bar u(t)$ given by Lemmas \ref{claim10} and \ref{MODULATION}.
Then,   for $\beta_{0}>0$ small enough, 
$$
\forall t\geq 0,\quad \bar \lambda(t)\leq 2.
$$
From the bound of $\bar u(t_m)$ in $H^{\frac \alpha 2}$, there exists
$\tilde u(0)\in H^{\frac \alpha 2}$ such that after possibly extracting a subsequence (still denoted by $(t_m)$)
$$
	\bar u(t_{m} , . + \rho(t_{m}) ) \rightharpoonup \tilde  u(0) 
	\quad \text{in $H^{\frac \alpha 2}$ as $m\to +\infty$.}
$$
Taking $\beta_{0}$ small enough, it is clear that $\tilde u(0)$ is close to $Q$ and in particular cannot be zero.
Let now $\tilde u(t)$ be the maximal solution of \eqref{dgBO} in $H^{\frac \alpha 2}$ corresponding to $\tilde u(0)$ given by Theorem \ref{TH1}.
We denote by $(-T_{1},T_{2})$ the maximal interval of existence of 
$\tilde u(t)$. Without a further analysis through Steps 2--4 below, we do not know if $\tilde u(t)$ is globally defined for $t>0$ or $t<0$.

\medskip

\noindent\textbf{Step 2.} First properties of the limiting problem.

\begin{lemma}\label{lemma13}
The following holds
\begin{equation}\label{eq1:le13}
	0< \beta(\tilde u(0)) \leq \beta_{0}
	\quad \text{and}\quad
E(\tilde u(0))   <0 .
\end{equation}
\end{lemma}

\begin{proof}[Proof of Lemma \ref{lemma13}]
Let
\begin{equation}\label{eq:104bis}
 v_{m}(x)= \bar u(t_{m},x+\rho(t_{m}))\rightharpoonup \tilde u(0)
\quad \text{in $H^{\frac \alpha 2}$ as $m\to +\infty$.}
\end{equation}
By weak convergence 
$$
\beta(\tilde u(0)) \leq \liminf_{{m\to +\infty}} \beta(v_{m}) < \beta_{0}.
$$
The positivity $\beta(\tilde u(0))>0$ is a consequence of the negativity of the energy of $\tilde u(0)$ and \eqref{eq:bGN}, which we prove now.

Let $\chi\in {\cal C}^\infty_0(\RR)$ such that $0\le \chi \le 1$, $\chi(x)=1$ if $|x|\le 1$ and $\chi(x)=0$ if $|x|\ge 2$.
  Let $\chi_{A}(x) = \chi(x/A)$, for $A>1$.
Then,
$$
E(v_{m}) =\|(\DD^{\frac \alpha 2}v_{m})\sqrt{\chi_{A}})\|_{L^2}^2 - \frac 1{(1+\alpha)(2\alpha +1)}\int |v_{m}\chi_A|^{2\alpha +2} + E(v_{m} (1-\chi_{A}) ) + R_{m,A} + \tilde R_{m,A},
$$
where
\begin{align*}
R_{m,A} & = \|\DD^{\frac \alpha 2}v_{m}\|_{L^2}^2 
- \|(\DD^{\frac \alpha 2}v_{m})\sqrt{\chi_{A}})\|_{L^2}^2
- \|\DD^{\frac \alpha 2}(v_{m}(1-\chi_{A}))\|_{L^2}^2,\\
\tilde R_{m,A} & = -\frac 1{(1+\alpha)(2\alpha +1)}\int |v_{m}|^{2\alpha +2}
(1-\chi_{A}^{2\alpha +2} - (1-\chi_{A})^{2\alpha+2}).
\end{align*}
First, we control the term $R_{m,A}$.
Note that from standard arguments, for all $u$,
\begin{align*}
\left| \| \DD^{\frac \alpha 2} (1-\chi_{A}) u \| - \|(1-\chi_{A}) \DD^{\frac \alpha 2} u \| \right|
&\leq C \| [\DD^{\frac \alpha 2}, (1- \chi_{A})] u\|=
C \| [\DD^{\frac \alpha 2}, \chi_{A}] u\| \leq \frac C{A^{\frac \alpha 2}}\|u\|,
\end{align*}
and so
$$
 \| \DD^{\frac \alpha 2} (1-\chi_{A}) u \|^2 \leq (1+A^{-\frac \alpha 2})  \|(1-\chi_{A}) \DD^{\frac \alpha 2} u \|^2
+  \frac C{A^{\frac \alpha 2 }}\|u\|^2.
$$
Combining these two estimates, we get
$$
R_{m,A} \geq - C A^{-\frac\alpha 2} \|v_m\|^2_{H^{\frac \alpha 2}} \geq - C A^{-\frac\alpha 2}
$$

Next, we control $\tilde R_{m,A}$. By weak convergence in $H^{\frac \alpha 2}$ and the properties of $\chi_A$, we have
$$
\lim_{m\to +\infty} \tilde R_{m,A} 
= - \frac 1{(1+\alpha)(2\alpha  +1)} \int |\tilde u(0)|^{2 \alpha +2} (1- \chi_A^{2\alpha +2}
- (1-\chi_A)^{2 \alpha +2}) = \tilde R_{A}.
$$
Moreover, from the definition of $\chi_A$, the following holds $\lim_{A\to +\infty} \tilde R_A=0$.

Finally, by \eqref{eq:bGN} (Gagliardo-Nirenberg with best constant), we have $E(v_m (1-\chi_A)) \geq 0$ since for $A$ large and $\beta_0$ small, for all $m$,
$\int v_m^2 (1-\chi_A)^2 \leq \frac 12 \int Q^2$.   

Therefore, 
\begin{align*}
0 > E(\bar u(0)) = E(v_m) & \geq \|(\DD^{\frac \alpha 2}v_{m})\sqrt{\chi_{A}})\|_{L^2}^2 - \frac 1{(1+\alpha)(2\alpha +1)}\int |v_{m}\chi_A|^{2\alpha +2} \\ &- C A^{-\frac\alpha 2} \|\tilde u(0)\|^2_{L^2} + \tilde R_{m,A}\end{align*}
and passing to the limit as $m\to +\infty$, we get
$$
0 > E(\bar u(0)) \geq \|(\DD^{\frac \alpha 2}\tilde u(0))\sqrt{\chi_{A}})\|_{L^2}^2 - \frac 1{(1+\alpha)(2\alpha +1)}\int |\tilde u(0)\chi_A|^{2\alpha +2}   - C A^{-\frac\alpha 2} \|\tilde u(0)\|^2_{L^2} + \tilde R_{A}.
$$
Finally, passing to the limit as $A \to +\infty$, we obtain
$
0 >E(\bar u(0)) \geq E(\tilde u(0)).$
\end{proof}

\begin{lemma}\label{lemma14}
For all $t\in (-T_{1},T_{2})$,
\begin{equation}\label{eq1:le14}
\bar u(t_{m}+t,\bar \rho(t_{m}) + .) \rightharpoonup \tilde u(t)\quad
\text{in $H^{\frac \alpha 2}(\RR)$ as $m\to +\infty$.}
\end{equation}
Moreover, if $\tilde \eta(t,x)$, $\tilde \lambda(t)$ and $\tilde \rho(t)$ are the parameters of the decomposition of $\tilde u(t,x)$, then for all $t\in (-T_{1},T_{2})$,
\begin{equation}\label{eq2:le14}
\bar \lambda(t_{m}+t) \to \tilde \lambda(t),\quad
\bar \rho(t_{m}+t) -\bar\rho(t_m)\to \tilde \rho(t).
\end{equation}
\end{lemma}
The first part of Lemma \ref{lemma14} follows from Theorem \ref{TH3}. By lemmas  \ref{lemma13}  and \ref{claim10}, $\tilde u(t)$ is close to $Q$ (up to scaling and translation) for all $t\in (-T_{1},T_{2})$, and we can apply lemma \ref{MODULATION} †to obtain a refined decomposition of $\tilde u(t)$ around $Q$, denoted by $\tilde \eta(t)$, $\tilde \lambda(t)$ and $\tilde \rho(t)$. Then \eqref{eq2:le14} follows from standard limiting and uniqueness arguments which we omit. See \cite{Me}, Lemma 8, Corollary 2 and references therein.

\medbreak

\noindent\textbf{Step 3.} Decay properties of the limiting problem by monotonicity properties.

\begin{lemma}\label{lemma15}
For all $t\in (-T_{1},T_{2})$, for all $x_{0}>1$,
\begin{equation}\label{eq1:le15}
\|\tilde u(t)\|_{L^2(|x-\tilde \rho(t)|\geq x_{0})}^2 \leq C |x_{0}|^{-\alpha}.
\end{equation}
\end{lemma}
\begin{proof}[Proof of Lemma \ref{lemma15}]
The main ingredient of the proof is Proposition \ref{PR:2} applied to $\bar u(t)$. 
Fix $\mu= \frac 12$,  $r= \frac { \alpha+1} 2$, $A$ large enough, and let $C_0=C(\frac 12 , r, A)>0$
be the constant given by  Proposition \ref{PR:2}.

First, we prove the decay estimates on the right.
Let $t\in (-T_1,T_2)$ and $m$ be such that $t_m+t>0$.
From \eqref{monotonicity1} applied to $\bar u$, $t_2=t_m+t$ and $t_1=0$, we have
\begin{align*}
& \int \bar u^2(t_m+t,x) \varphi_A(x-\bar \rho(t_m+t) -x_0) dx \\& \leq 
\int \bar u^2(0,x) \varphi_A(x-\bar \rho(0) - \tfrac 12 (\bar \rho(t_m+t)  - \bar \rho(0)) -x_0) dx + \frac {C_0}{x_0^{2r-1}}.
\end{align*}
Thus, passing to the limit as $m\to +\infty$, using $\bar\rho(t_m+t)\to +\infty$ when $m\to +\infty$, we have    
\begin{equation}\label{eq:tbu}
\limsup_{m\to +\infty} \int \bar u^2(t_m+t,x) \varphi_A(x-\bar \rho(t_m+t) -x_0) dx \leq \frac {C_0}{x_0^{2r-1}}.
\end{equation}
It follows from the previous estimate and Lemma \ref{lemma14} that
$$
	\int \tilde u^2(t,x)  \varphi_A(x-\tilde \rho(t) -x_0) dx  \leq \frac {C_0}{x_0^{2r-1}}.
$$

Second, we prove decay estimate on the left. Let $t\in (-T_1,T_2)$ and let $m,$ $m'$ be such that
$t_m > t_{m'} +t$. Using \eqref{monotonicity2}, we obtain
\begin{align*}
& \int \bar u^2(t_m,x) \varphi_A(x- \bar \rho(t_m) + \tfrac 12 (\bar \rho(t_m) - \bar \rho(t_{m'}+t) ) +x_0) dx\\
& \leq \int \bar u^2 (t_{m'}+t,x) \varphi_A(x- \bar \rho(t_{m'} + t) +x_0) dx + \frac {C_0}{x_0^{2 r-1}}.
\end{align*}
By Lemma \ref{lemma14}, we have on the one hand, for $m'$ fixed,   
$$
\liminf_{m\to +\infty} \int \bar u^2(t_m,x) \varphi_A(x- \bar \rho(t_m) + \tfrac 12 (\bar \rho(t_m) - \bar \rho(t_{m'}) ) +x_0) dx \geq \int \tilde u^2(t),
$$
and on the other hand, using \eqref{eq:tbu},
\begin{align*}
& \limsup_{m' \to +\infty} \int \bar u^2 (t_{m'}+t,x) \varphi_A(x- \bar \rho(t_{m'} + t) +x_0) dx \\
&\leq \int \tilde  u^2 (t,x) \varphi_A(x- \tilde  \rho( t) +x_0) dx + \frac {C_0}{x_0^{2r-1}}.
\end{align*}
It follows that
$$
 \int \tilde  u^2 (t,x) (1- \varphi_A(x- \tilde  \rho( t) +x_0) ) dx \leq  \frac {2 C_0}{x_0^{2r-1}}.
$$
Lemma \ref{lemma15} is now proved.
\end{proof}

\medskip

\noindent\textbf{Step 4.} Conclusion of the proof by rigidity properties.
From \eqref{eq1:le13} and Lemma \ref{claim10}, we have
\begin{equation}\label{eq:truc}
|\tilde \lambda(0) -1|  \leq \delta(\beta_0),\quad \text{where}\quad
\lim_{\beta_0 \to 0}\delta(\beta_0)   =0. 
\end{equation}

We claim the following lemma to be used as a bootstrap argument on the behavior of $\tilde \lambda(t)$.

\begin{lemma}\label{cl:9}
	Assume further  that for $-T_1<-t_1<0<t_2<T_2$,
	\begin{equation}\label{eq1:cl9}
		\forall t\in (-t_1,t_2),\quad | \tilde \lambda(t) - 1 | \leq \frac  12,
	\end{equation}	
	then for some $\epsilon>0$,
	\begin{equation}\label{eq2:cl9}
		\forall t\in (-t_1,t_2),\quad
			\tilde \eta(t)\in L^1(\RR) \quad \text{and} \quad
			\int |\tilde \eta(t,x)| dx \leq  C \beta_0^{\epsilon}.
	\end{equation}
\end{lemma}

Assuming Lemma \ref{cl:9}, we finish the proof of Theorem \ref{TH2}.
Using the invariant 
$$
\forall t\in (-T_1,T_2),\quad \int \tilde u(t)= \int \tilde u(0)
$$
and Lemma \ref{cl:9}, we prove that the solution $\tilde u(t)$ is global (i.e. $T_1=T_2=\infty$) and
$$
|\tilde \lambda(0) -1|  \leq \tilde \delta(\beta_0),\quad
\lim_{\beta_0 \to 0}\tilde \delta(\beta_0) =0 .
$$

By \eqref{eq1:cl9}, \eqref{eq2:cl9}, for all $t\in (-t_1,t_2)$, we have
$$
\left|\int \tilde u(t) - \int Q_{\tilde \lambda(t)}  \right| \leq C \beta_0^\epsilon,
$$
and so since $\int Q_\lambda=\lambda^\frac1\alpha\int Q$, 
\begin{equation}\label{eq4:pth2}
\left| \tilde \lambda(t)^{\frac1\alpha} - \tilde \lambda(0)^\frac1\alpha\right|\leq 
\left| \int Q_{\tilde \lambda(0)} - \int Q_{\tilde \lambda(t)}\right|
 \leq C\beta_0^\epsilon,
\end{equation}
Therefore, by a standard continuity argument, \eqref{eq:truc}, \eqref{eq1:cl9} and thus \eqref{eq4:pth2} are satisfied on 
$(-T_1,T_2)$. Thus, $\tilde u(t)$ is bounded on $(-T_1,T_2)$ in $H^\frac\alpha2$, which proves that $T_1=T_2=\infty$, and means that $\tilde u(t)$ is global. Moreover, \eqref{eq4:pth2} is satisfied for all $t\in \RR$.
By Proposition \ref{PR:4}, $\tilde u$ has to be a soliton but this is a contradiction with $E(\tilde u(0))<0$, since the energy of a soliton is zero.  
This concludes the proof of Theorem \ref{TH2} assuming Lemma \ref{cl:9}.
Thus, we only have to prove Lemma~\ref{cl:9}.

\begin{proof}[Proof of Lemma \ref{cl:9}]
We prove the result for $t\in (0,t_2)$, the proof being the same for negative times.
Let $0<\epsilon<\frac 12 (\alpha-1)$ small to be chosen later. As long as \eqref{eq1:cl9} is satisfied, we have by Lemma \ref{lemma15},
$$
x_0^{\epsilon}  \int_{|x|>x_0} \tilde u^2(t,x+\tilde \rho(t)) dx \leq C |x_0|^{-\alpha+\epsilon}.
$$
Integrating this estimate in $x_0$ and using Fubini theorem, we obtain
\begin{equation}\label{eq2:pth2}
\int |x|^{1+\epsilon} \tilde u^2(t,x+\tilde \rho(t)) dx \leq C.
\end{equation}
By the definition of $\tilde \eta(t)$ and the decay properties of $Q$, as long as \eqref{eq1:cl9} is satisfied, we
obtain
\begin{equation}\label{eq3:pth2}
\int |x|^{1+\epsilon} \tilde \eta^2(t,x) dx \leq C.
\end{equation}
In particular, by Holder inequality,
\begin{align*}
\int |\tilde \eta(t)| & \leq \|\tilde \eta\|_{L^\infty}^\epsilon \int |\tilde \eta(t)|^{1-\epsilon} \\ &\leq 
 \|\tilde \eta\|_{L^\infty}^\epsilon \left( \int |\tilde \eta(t)|^2 (1+|x|)^{\frac 1{(1-\epsilon)^2}}\right)^{\frac {1-\epsilon}2}
 \left(\int (1+|x|)^{-\frac 1{1-\epsilon^2}} \right)^{\frac {1+\epsilon}2}\\
& \leq   C  \|\tilde \eta\|_{L^\infty}^\epsilon
\end{align*}
and the result follows from
$
\|\tilde \eta\|_{L^\infty}\leq \|\tilde \eta\|_{H^{\frac \alpha 2}}
$
and Lemma \ref{MODULATION}.
\end{proof}

\appendix
\section{Appendix}
In the appendix, we gather the proof of standard results for reader's convenience.

\begin{lemma}\label{LE:decay} Let $r>\frac 12$ and $\alpha>-1$. Let
$$
g(x)=g_{\alpha,r}(x)=\DD^{\alpha} \left(\frac 1 {\langle x\rangle^{2r}}\right)
\quad\text{and}\quad 
h(\xi)=\int e^{-i x\xi} \frac 1 {\langle x\rangle^{2r}} dx
\quad \text{so that} \quad  \hat g(\xi)=|\xi|^\alpha h(\xi) .
$$
Then
	\begin{itemize}
\item[(i)]  There exists $C>0$ such that
$$
\left|g_{\alpha,r}\right(x)| \leq \frac C{\langle x \rangle^{\alpha+1}}.
$$
\item[(ii)] The function $h$ is continuous, and for any $M>0$, there exists $C_M>0$ such that $|h(\xi)|\leq \frac {C_M}{\langle \xi\rangle^M}$.

Moreover, $h\in C^{\infty}(\RR\setminus \{0\})$ and for all $\beta\in \NN, \ q>0,$ there exists
$C_{\beta,q}>0$ such that
$$
\left|\partial_\xi^\beta h(\xi)\right|
\leq \frac {C_{\beta,q}}{|\xi|^\beta \langle \xi \rangle^q}.
$$
\end{itemize}
\end{lemma}

\begin{proof} The proof is standard. Clearly $h$ is a continuous and bounded function. By integration by part we have,
\begin{equation}
 \begin{array}{lll}
 (i\xi)^Nh(\xi)&=&\displaystyle \int (-\px)^N(e^{-ix\xi})\dfrac{1}{\langle x \rangle^{2r}}dx\\[12pt]
&=&\displaystyle\int e^{-ix\xi}(\px)^N\left(\dfrac{1}{{\langle x \rangle}^{2r}}\right) dx
\end{array}
\end{equation}

We have $\displaystyle |(\px)^N\left(\dfrac{1}{{\langle x \rangle}^{2r}}\right)|\le \frac{C}{{\langle x \rangle}^{2r+N}}$ which is an integrable function and so $\xi^Nh(\xi)$ is bounded. This gives the first part of \textit{(ii)}.

 Let $\chi\in {\cal C}^\infty_0(\RR)$ such that $0\le \chi \le 1$, $\chi(\xi)=1$ if $|\xi|\le 1$ and $\chi(\xi)=0$ if $|\xi|\ge 2$, we set $h_N(\xi)=\int e^{-ix\xi}\frac{1}{\langle x \rangle^{2r}}\chi(\frac{x}{N})dx$, $h_N\to h$ uniformly and in particular in ${\cal D'}$, then  $\pxi^\beta h_N \to\pxi^\beta h$ in ${\cal D'}$. Let $M>0$ to be fixed below, we have for some non important constants $C$,
\begin{equation}
 \begin{array}{lll}
 \xi^M\pxi^\beta h_N(\xi)&=& \displaystyle C_\beta\int e^{-ix\xi}\frac{\xi^M x^\beta}{\langle x \rangle^{2r}}\chi(\frac{x}{N})dx\\[12pt]
&=&  \displaystyle C_{\beta, M}\int \px^M (e^{-ix\xi})\frac{x^\beta }{\langle x \rangle^{2r}}\chi(\frac{x}{N})dx\\[12pt]
&=&  \displaystyle \sum_{M_1+M_2=M} C_{\beta, M_1,M_2}\int e^{-ix\xi} \px^{M_1} (\frac{x^\beta }{\langle x \rangle^{2r}})\frac{1}{N^{M_2}}\px^{M_2} (\chi)(\frac{x}{N})dx
\end{array}
\end{equation}
We have $|\px^{M_1} (\frac{x^\beta }{\langle x \rangle^{2r}})|\le \frac{C}{\langle x \rangle^{2r-\beta+M_1}}$.

If $M_2\ge 1$, the integral is restricted to $N\le |x|\le 2N$ and we have
\begin{equation}
 |\int e^{-ix\xi} \px^{M_1} (\frac{x^\beta }{\langle x \rangle^{2r}})\frac{1}{N^{M_2}}\px^{M_2} (\chi)(\frac{x}{N})dx|\le \frac{C}{N^{2r-\beta+M-1}}
\end{equation}
then these terms goes to 0 if $2r-\beta+M-1>0$.

If $M_2=0$, $\frac{1}{\langle x \rangle^{2r-\beta+M}}$ is integrable if $2r-\beta+M-1>0$. This implies that $\xi^M\pxi^\beta h_N(\xi) \to C_{\beta,M}\int e^{-ix\xi} \px^{M} (\frac{x^\beta }{\langle x \rangle^{2r}})dx$ uniformly and since $\xi^M\pxi^\beta h_N \to\xi^M\pxi^\beta h$ in ${\cal D'}$, we obtain 
\begin{equation}
\xi^M \pxi^\beta h(\xi)= C_{\beta,M}\int e^{-ix\xi} \px^{M} (\frac{x^\beta }{\langle x \rangle^{2r}})dx
\end{equation}
If we take $M=\beta$ we obtain the second part of the estimate  \textit{(ii)} if $|\xi|\le 1$. If we take $M=\beta+q$ we obtain the second part of the estimate  \textit{(ii)} if $|\xi|\ge 1$.

Now, we prove  \textit{(i)}. We set $g_{\alpha,r}(x)=\frac{1}{2\pi}(g_1(x)+g_2(x))$ where 
\begin{equation}
 \begin{array}{lll}
g_1(x)&=&\int e^{ix\xi}|\xi|^\alpha h(\xi)(1-\chi(\xi))d\xi\\[10pt]
g_2(x)&=&\int e^{ix\xi}|\xi|^\alpha h(\xi)\chi(\xi)d\xi
\end{array}
\end{equation}
Following  \textit{(ii)}, $|\xi|^\alpha h(\xi)$ is integrable, thus $g_1$ is continuous and bounded and for all $M>0$,
\begin{equation}\label{eq1:LE:decay}
|\pxi^\beta(|\xi|^\alpha h(\xi)(1-\chi(\xi)))|\le \frac{C}{\langle \xi \rangle^{M}}
\end{equation}
moreover, by integration by part, we have
\begin{equation}\label{eq2:LE:decay}
 x^\beta g_1(x)=\int i^\beta e^{ix\xi}\pxi^\beta(|\xi|^\alpha h(\xi)(1-\chi(\xi)))d\xi
\end{equation}
\eqref{eq1:LE:decay} and \eqref{eq2:LE:decay}  give that $x^\beta g_1(x)$ bounded for all $\beta$.

To estimate $g_2$ we assume $x\ge 1$, the case $x\le -1$ follows by the same way. We set $x\xi=\sigma$. We have 
\begin{equation}
 \begin{array}{lll}
g_2(x)&=&x^{-\alpha -1} \int e^{i\sigma}|\sigma|^\alpha h(\frac{\sigma}{x})\chi(\frac{\sigma}{x})d\sigma = x^{-\alpha -1}(k_1(x)+k_2(x))\textrm{ where}\\[10pt]
k_1(x)&=& \int e^{i\sigma}\chi(\sigma)|\sigma|^\alpha h(\frac{\sigma}{x})\chi(\frac{\sigma}{x})d\sigma \\[10pt]
k_2(x)&=& \int e^{i\sigma}(1-\chi(\sigma))|\sigma|^\alpha h(\frac{\sigma}{x})\chi(\frac{\sigma}{x})d\sigma 
\end{array}
\end{equation}
Obviously $k_1$ is bounded. By integration by part we have
\begin{equation}
 \begin{array}{lll}
k_2(x)&=& \displaystyle \int (-i\partial_\sigma)^N(e^{i\sigma})(1-\chi(\sigma))|\sigma|^\alpha h(\frac{\sigma}{x})\chi(\frac{\sigma}{x})d\sigma  \\[10pt]
&=&\displaystyle \sum_{N_1,N_2,N_3}\int e^{i\sigma}\partial_\sigma^{N_1}( (1-\chi(\sigma))|\sigma|^\alpha  )\frac{1}{x^{N_2+N_3}}(\partial_\sigma^{N_2}h)(\frac{\sigma}{x})
(\partial_\sigma^{N_3}\chi )(\frac{\sigma}{x})d\sigma
\end{array}
\end{equation}
We have $|\partial^N_\sigma((1-\chi(\sigma))|\sigma|^\alpha)|\le \frac{C}{\langle \sigma\rangle^{N-\alpha}}$.

If $N_3\ge 1$, $x\le |\sigma|\le 2x$ and we obtain
\begin{equation}
 \int |\partial_\sigma^{N_1}( (1-\chi(\sigma))|\sigma|^\alpha  )\frac{1}{x^{N_2+N_3}}(\partial_\sigma^{N_2}h)(\frac{\sigma}{x})
(\partial_\sigma^{N_3}\chi )(\frac{\sigma}{x})|d\sigma\le \frac{C}{\langle x\rangle^{N-\alpha-1}}
\end{equation}
which is bounded if $N\ge \alpha+1$.

If $N_3=0$, following \textit{(ii)}, we have
\begin{equation}
 \int |\partial_\sigma^{N_1}( (1-\chi(\sigma))|\sigma|^\alpha  )\frac{1}{x^{N_2}}(\partial_\sigma^{N_2}h)(\frac{\sigma}{x})\chi(\frac{\sigma}{x})|d\sigma\le \int\frac{C}{{\langle \sigma\rangle^{N-\alpha}}}d\sigma
\end{equation}
which is bounded for $N$ large enough. This proves  \textit{(i)}.
\end{proof}

\begin{lemma}\label{LE:decay-2}
 Let $p(\xi)$ an homogeneous function of degree $\beta>-1$. Let $\chi\in {\cal C}^\infty_0(\RR)$ such that $0\le \chi \le 1$, $\chi(\xi)=1$ if $|\xi|\le 1$ and $\chi(\xi)=0$ if $|\xi|\ge 2$.
 Let 
\begin{equation}
k(x)=\frac1{2\pi}\int e^{ix\xi}p(\xi)\chi(\xi)d\xi 
\end{equation}
then for all $q\in\NN$, there exists $C_q>0$ such that for all $x\in\RR$
\begin{equation}
 |\px ^qk(x)|\le \frac{C_q}{\langle x\rangle^{\beta+q+1}}
\end{equation}
\end{lemma}

\begin{proof} The proof is standard. We have $\px^qk(x)=\frac{1}{2\pi}\int e^{ix\xi}(i\xi)^qp(\xi)\chi(\xi)d\xi$ and as $(i\xi)^qp(\xi)$ is homogeneous of degree $\beta+q$, it is sufficient to prove Lemma \ref{LE:decay-2} for $q=0$. We shall prove the estimate for $x\ge 1$, the case $x\le-1$ follows by the same way.
We set $y=x\xi$ in integral, we have $\int e^{ix\xi}(i\xi)^qp(\xi)\chi(\xi)d\xi=\frac{1}{x^{\beta+1}}\int e^{iy}p(y)\chi(\frac{y}{x})dy$. Lemma  \ref{LE:decay-2} will be proved if we prove that $\int e^{iy}p(y)\chi(\frac{y}{x})dy$ is bounded. We set $J_1=\int e^{iy}p(y)\chi(y)\chi(\frac{y}{x})dy$ and $J_2=\int e^{iy}p(y)(1-\chi(y))\chi(\frac{y}{x})dy$. We remark that $J_1$ does not depend of $x$ if $x$ large enough. We prove that $J_2$ is bounded by integration by part. For $N>\beta +1$ we have $\partial_y^Ne^{iy}=i^Ne^{iy}$ and by integration by part we have,
\begin{equation}
 J_2=\sum_{N_1+N_2+N_3=N}C_{N_1,N_2,N_3}\int e^{iy}\partial_y^{N_1}p(y)\partial_y^{N_2}(1-\chi(y))\frac{1}{x^{N_3}}(\partial_y^{N_3}\chi)(\frac{y}{x})dy
\end{equation}
If $N_2\ge 1$ we integrate on compact domain and these integrals are bounded.

If $N_3\ge 1$ in these integrals we have $x\le |y|\le 2x$ and
\begin{equation}
 |e^{iy}\partial_y^{N_1}p(y)(1-\chi(y))\frac{1}{x^{N_3}}(\partial_y^{N_3}\chi)(\frac{y}{x})|\le C|x|^{\beta-N_1-N_2}\le C |x|^{-1}
\end{equation}
then these integrals are bounded.

If $N_2=N_3=0$
\begin{equation}
 |e^{iy}\partial_y^{N}p(y)(1-\chi(y))\chi(\frac{y}{x})|\le C|y|^{\beta-N}(1-\chi(y))
\end{equation}
and this function is integrable. This proves Lemma \ref{LE:decay-2}.
\end{proof}

\end{document}